%% file: notes.tex
\documentclass{amsart}

\usepackage[utf8]{inputenc}
\usepackage[T1]{fontenc}
\usepackage[hidelinks, unicode]{hyperref}
\usepackage{cleveref}
\usepackage{amsthm}
\usepackage{amssymb}
\usepackage{bm}
\usepackage{quiver}

\usepackage{import}

\usepackage{graphicx}

\usepackage{tikzit}
\input{tikzstyles.tikzstyles}
\usetikzlibrary{decorations.markings}

\usepackage{caption}
\usepackage{xparse}

\usepackage{nicematrix,booktabs}

\numberwithin{figure}{section}

\newtheorem{theorem}{Theorem}[section]
\newtheorem*{theorem*}{Theorem}

\newtheorem{lemma}[theorem]{Lemma}
\newtheorem*{lemma*}{Lemma}
\newtheorem{proposition}[theorem]{Proposition}
\newtheorem*{proposition*}{Proposition}

\newtheorem{conjecture}[theorem]{Conjecture}

\newtheorem{corollary}[theorem]{Corollary}
\newtheorem*{corollary*}{Corollary}

\theoremstyle{remark}
\newtheorem{remark}[theorem]{Remark}
\newtheorem{notation}[theorem]{Notation}
\newtheorem{terminology}[theorem]{Terminology}
\newtheorem*{question*}{Question}
\theoremstyle{definition}
\newtheorem{definition}[theorem]{Definition}
\newtheorem{propositiondefinition}[theorem]{Proposition-Definition}
\newtheorem*{definition*}{Definition}

\newcommand{\Cat}{\mathcal{C}}

\newcommand{\cE}{\mathcal{E}}
\newcommand{\ocE}{\overline{\mathcal{E}}}

\newcommand{\cO}{\mathcal{O}}
\newcommand{\cI}{\mathcal{I}}

\newcommand{\cM}{\mathcal{M}}

\newcommand{\Kbar}{\overline{K}}

\newcommand{\VL}{\mathrm{Vect}_\Lambda}
\newcommand{\VLone}{\mathrm{Vect}_{\Lambda_1}}
\newcommand{\VLtwo}{\mathrm{Vect}_{\Lambda_2}}
\newcommand{\CN}{\mathrm{C}_\Nu}

\newcommand{\GMod}{\mathfrak{gmod}}
\newcommand{\BraidFun}{\mathfrak{braidfun}}

\newcommand{\Braid}{\mathfrak{Braid}}
\newcommand{\Rib}{\mathfrak{Rib}}

\newcommand{\Modcat}{\mathfrak{Mod}}

\newcommand{\lra}{\longrightarrow}

\newcommand{\ra}{\rightarrow}

\newcommand{\Nu}{\mathcal{V}}

\newcommand{\PMod}[1]{\mathrm{PMod}(#1)}

\newcommand{\tPMod}[1]{\mathrm{PM\tilde{o}d}(#1)}

\newcommand{\PB}{\mathrm{PB}}

\protected\def\myphantom#1{\vphantom{#1}}

\newcommand{\mysecondleftidx}[3]{{\myphantom{#2}}#1#2#3}

\newcommand{\Mrpo}[2]{\mysecondleftidx{^1}{\overline{\mathcal{M}}}{_{0,{#1}}}(#2)}

\newcommand{\otMgb}[2]{\widetilde{\overline{\mathcal{M}}}_{#1}^{#2}}
\newcommand{\bbV}{\mathbb{V}}

\newcommand{\Mg}[2]{\mathcal{M}_{{#1},{#2}}}
\newcommand{\Mgb}[2]{\mathcal{M}_{#1}^{#2}}

\newcommand{\tMgb}[2]{\widetilde{\mathcal{M}}_{#1}^{#2}}

\newcommand{\oMg}[2]{\overline{\mathcal{M}}_{{#1},{#2}}}

\newcommand{\Mgrb}[3]{\overline{\mathcal{M}}_{#1}^{#2}(#3)}

\newcommand{\Mgrbt}[4]{\overline{\mathcal{M}}_{#1}^{#2}(#3,#4)}

\newcommand{\Conf}[2]{\mathrm{Conf}_{#1}(#2)}

\newcommand{\Br}[1]{\mathrm{B}\mu_{#1}}

\newcommand{\End}[1]{\mathrm{End}(#1)}

\newcommand{\myim}[1]{\mathrm{im}\left({#1}\right)}

\newcommand{\Aut}[1]{\mathrm{Aut}(#1)}

\newcommand{\Z}{\mathbb{Z}}
\newcommand{\N}{\mathbb{N}}

\newcommand{\C}{\mathbb{C}}
\newcommand{\Q}{\mathbb{Q}}

\newcommand{\id}{\mathrm{id}}

\newcommand{\ul}{\underline{\lambda}}

\newcommand{\ua}{\underline{a}}

\newcommand{\RS}[2]{\mathcal{RS}({#1},{#2})}
\newcommand{\RSo}[1]{\mathcal{RS}({#1})}

\newcommand{\A}{\mathbb{A}}

\newcommand{\Sp}{\mathrm{Sp}}

\DeclareMathOperator{\Gal}{Gal}
\DeclareMathOperator{\Spec}{Spec}

\title{Conformal blocks are quasi-geometric}
\author{Pierre Godfard}

\begin{document}

\begin{abstract}
  We prove that the bundles with flat connections on configuration spaces associated to braided fusion categories,
  as well as the bundles with flat connections on moduli spaces of curves (conformal blocks) associated to modular fusion categories,
  are defined over number fields. The proof relies on Ocneanu rigidity. This result answers a conjecture of Etingof and Varchenko.
  Furthermore, we show that for a fixed braided or modular category, all the associated bundles with flat connections and their compatibilities
  (i.e., the braided or modular functor) can be defined over the same number field.
\end{abstract}

\maketitle


\section{Introduction}


In \cite[4.26]{etingofPeriodicQuasimotivicPencils2023}, Etingof and Varchenko conjectured that bundles with flat connections
associated to braided fusion and modular fusion categories are defined over $\overline{\Q}$. This is expected as the conjecture
is known in a large class of examples (for instance those constructed from affine Lie algebras), and that
the monodromy of these bundles is already known to be defined over $\overline{\Q}$ in the general case.

In this paper, we provide a proof of their conjecture based on Ocneanu rigidity.


\subsection{Bundles with flat connection associated to modular, ribbon and braided fusion categories}


A braided fusion category $\Cat$ gives rise to a family of complex finite dimensional pure braid group representations.
More precisely, for each choice of $n\geq 1$ and objects $y,x_1,\dotsc,x_n$ in $\Cat$, there is a representation:
\begin{equation*}
  \PB_n\lra \mathrm{GL}\left(\hom_\Cat(y,x_1\otimes\dotsb\otimes x_n)\right)
\end{equation*}
defined using the associator and the braiding in $\Cat$. Moreover, these representations have some compatibilities.
For instance, given objects $y,z,x_1,\dotsc,x_n,w_1,\dotsc,w_m$ in $\Cat$, the partial composition map:
\begin{multline*}
  \hom_\Cat(y,z\otimes x_1\otimes\dotsb\otimes x_n)\otimes\hom_\Cat(z,w_1\otimes\dotsb\otimes w_m)\\
  \lra \hom_\Cat(y,w_1\otimes\dotsb\otimes w_m\otimes x_1\otimes\dotsb\otimes x_n)
\end{multline*}
is equivariant for the action of $\PB_{1+n}\times \PB_m\subset \PB_{m+n}$, where the inclusion of groups is obtained by splitting the $1^\mathrm{st}$
strand of $\PB_{1+n}$ into $m$ strands.
These pure braid group representations and their compatibilities fit into a structure called topological braided functor (\Cref{braidedfunctor}).

As the pure braid group $\PB_n$ is the fundamental group of the ordered configuration space%
\footnote{The quotient is by translations and does not change the fundamental group. We choose to work with the quotients
to simplify statements about compatibilities.} $\Conf{n}{\C}/\C$,
by the Riemann-Hilbert correspondence, the representation on $\hom_\Cat(y,x_1\otimes\dotsb\otimes x_n)$ gives rise to
a regular-singular algebraic bundle with flat connection $(\Nu(y;x_1,\dotsc,x_n),\nabla)$ over $\Conf{n}{\C}/\C$.
These bundles and their generalizations to higher genus (see below) are called \emph{conformal blocks}.
The compatibilities mentioned above then have algebraic interpretations, and fit into the structure of a geometric braided functor (\Cref{braidedfunctor}).

Because the representation of $\PB_n$ on $\hom_\Cat(y,x_1\otimes\dotsb\otimes x_n)$ factors through
$\PB_n/\{T^r\:\mid\: T\text{ Dehn twist}\}$ for some $r$ depending only on $\Cat$, the flat bundles $(\Nu(y;x_1,\dotsc,x_n),\nabla)$
can be extended to a suitable stack $\Mrpo{n}{r}$ compactifying $\Conf{n}{\C}/\C$.

Pure braid groups are genus $0$ mapping class groups, and when $\Cat$ is assumed to have a ribbon structure and to satisfy a condition called modularity,
the Reshetikhin-Turaev construction applied to $\Cat$ provides representations of mapping class group of any genus
(\cite{reshetikhinInvariants3manifoldsLink1991, turaevQuantumInvariantsKnots2016}).
More precisely, for each genus $g\geq 0$ and $x_1,\dotsc,x_n$ objects in $\Cat$, one has a representation:
\begin{equation*}
  \tPMod{S_g^n}\lra \mathrm{GL}\left(\hom_\Cat (\mathbf{1},ad^{\otimes g}\otimes x_1\otimes\dotsb\otimes x_n)\right)
\end{equation*}
where $\tPMod{S_g^n}$ is a central extension by $\Z$ of the pure mapping class group of the compact surface $S_g^n$ of genus $g$ with $n$ boundary components,
and $\mathbf{1}$ and $ad$ are specific objects of $\Cat$.
These representations have compatibilities with regards to gluings of surfaces along boundary components
and fit into the structure of a topological modular functor
(\Cref{definitiongeometricmodularfunctor}).

As fundamental groups of Deligne-Mumford-Knudsen moduli spaces of curves
are mapping class group, the Riemann-Hilbert correspondence applied to the above representations yields
regular-singular bundles with flat connections $(\Nu_g(x_1,\dotsc,x_n),\nabla)$
on smooth stacks $\tMgb{g}{n}$ (\cite[6.4.1]{bakalovLecturesTensorCategories2000}, \Cref{definitiongeometricmodularfunctorBK}).
As in the case of configuration spaces,
these connections extend to suitable stacks $\Mgrbt{g}{n}{r}{s}$ compactifying $\tMgb{g}{n}$,
and fit together into a structure called geometric modular functor (\Cref{definitiongeometricmodularfunctor}). The following table recaps the situation
and also incorporates the case of ribbon fusion categories that we did not mention above.

\begin{center}
  \begin{NiceTabular}{ |m{0.12\textwidth}|m{0.12\textwidth}|m{0.25\textwidth}|m{0.15\textwidth}|m{0.2\textwidth}| }
    \toprule
   Category:      & Functor:                  & Representations of:                                             & Regular-singular flat bundles on:           & ...or flat bundles on the compactification: \\[0.25cm] 
   \midrule
   Braided fusion & Braided functor           & Pure braid groups $\PB_n$                                       & $\Conf{n}{\C}/\C$                           & $\Mrpo{n}{r}$\phantom{aaa} (rk. \ref{remarkconf})\\[0.25cm] 
   \midrule
   Ribbon fusion  & Genus $0$ modular functor & Genus $0$ mapping class groups $\PMod{S_0^n}$                   & $\Mgb{0}{n}$\phantom{aaaa} (rk. \ref{remarkBKgenuszero}) & $\Mgrb{0}{n}{r}$\phantom{aaaaa} (sec. \ref{twistedmodulispaces})\\[0.25cm] 
   \midrule
   Modular fusion & Modular functor           & Central $\Z$-extensions $\tPMod{S_g^n}$ of mapping class groups & $\tMgb{g}{n}$\phantom{aaa} (sec. \ref{sectionBK})        & $\Mgrbt{g}{n}{r}{s}$\phantom{aaa} (sec. \ref{twistedmodulispaces})\\[0.25cm]  
   \bottomrule
  \end{NiceTabular}
\end{center}


\subsection{Quasi-geometricity}


The most studied family of modular categories and associated modular functors is constructed from simple Lie algebras.
More precisely, for each simple complex Lie algebra $\mathfrak{g}$, integer $\ell\geq 1$ called level
and root of unity $\zeta$ of order $r$ (prescribed by $\mathfrak{g}$ and $\ell\geq 1$), one has a modular category $\Cat_{\mathfrak{g},\zeta}$
and an associated modular functor $\Nu_{\mathfrak{g},\zeta}$.
The category $\Cat_{\mathfrak{g},\zeta}$ is a semi-simplification of a representation category of the quantum group $\mathrm{U}_\zeta(\mathfrak{g})$
(see \cite[3.3]{bakalovLecturesTensorCategories2000}),
while, for $\zeta=e^{\frac{2\pi i}{r}}$, the algebraic bundles with connection $((\Nu_{\mathfrak{g},\zeta})_g(x_1,\dotsc,x_n),\nabla)$
can be defined using the representation theory of the affine Lie algebra
associated to $\mathfrak{g}$ (see \cite[7.]{bakalovLecturesTensorCategories2000}).

From their quantum group construction, one sees that the categories $\Cat_{\mathfrak{g},\zeta}$ are defined on a cyclotomic field,
and hence so are the associated mapping class group representations.
Moreover, the affine Lie algebra construction of the 
$((\Nu_{\mathfrak{g},\zeta})_g(x_1,\dotsc,x_n),\nabla)$ can be done over $\Q$. Hence, for $\zeta=e^{2i\pi/r}$,
the representations/connections are \emph{quasi-geometric}
in the terminology of Etingof-Varchenko.

\begin{definition}[Quasi-geometric connection {\cite[4.18]{etingofPeriodicQuasimotivicPencils2023}}]
  Let $X$ be a smooth proper Deligne-Mumford stack over $\overline{\Q}$, $D\subset X$ a divisor with normal crossings, and $U=X\setminus D$.
  An algebraic regular-singular bundle with flat connection $(\cE,\nabla)$ on $(X_\C,D_\C)$ is \emph{quasi-geometric}
  if it can be defined over $(X,D)$, i.e. over $\overline{\Q}$,
  and if the monodromy representation of $\pi_1(U_\C^\mathrm{an})$ associated to $(\cE,\nabla)$ is conjugate to a representation
  with coefficients in $\overline{\Q}$.
\end{definition}

The notion of quasi-geometricity can be motivated by Simpson's geometricity conjecture explained in the next subsection.
In their work on periodic pencils of flat connections, Etingof and Varchenko conjectured the following:

\begin{conjecture}[{\cite[4.26]{etingofPeriodicQuasimotivicPencils2023}}]\label{conjectureEV}
  For any braided fusion category, the associated flat bundles on configuration spaces $\Conf{n}{\C}$ are quasi-geometric.
  For any modular category, the associated flat bundles on moduli spaces of curves $\tMgb{g}{n}$ are quasi-geometric.
\end{conjecture}

One of their motivation for supporting this conjecture is a result called Ocneanu rigidity (\cite[2.28]{etingofFusionCategories2005}),
which says, in particular, that braided fusion and modular fusion categories have no non-trivial deformations.
This structural rigidity implies that such categories can always be defined on a number field, and thus that the associated
braid group/mapping class group representations are defined over a number field. This already goes halfway to proving quasi-geometricity,
and they suggest that Ocneanu rigidity should also imply that the associated connections are definable over $\overline{\Q}$.
In this paper, we complete this picture and prove their conjecture.

\begin{theorem*}[{\ref{mainresult}}]
  Let $\Nu$ be a geometric modular functor over $\C$, and let $(r,s)$ be a level for it.
  Then there exists a number field $K$
  such that $\Nu$ can be defined over $K$ in the sense of \Cref{definitiongeometricmodularfunctor},
  i.e. each bundle with flat connection and gluing, forgetful and permutation isomorphism defining $\Nu$
  can be defined over the twisted moduli spaces of curves $\Mgrbt{g}{n}{r}{s}_K$ over the field $K$.

  The same results hold for genus $0$ geometric modular functors and geometric braided functors.
\end{theorem*}

\begin{remark}
  The result in particular applies to the connections $((\Nu_{\mathfrak{g},\zeta})_g(x_1,\dotsc,x_n),\nabla)$ associated to Lie algebras
  for all values of $\zeta$, not just $\zeta=e^{2i\pi/r}$.
\end{remark}

We also prove the same result for the regular-singular bundles with flat connections on the spaces $\tMgb{g}{n}$, $\Mgb{0}{n}$ and $\Conf{n}{\C}/\C$,
associated to modular, ribbon and braided fusion categories respectively, see \Cref{mainresultBK}.
In fact, a large portion of this paper (\Cref{sectionBK} and \Cref{appendixregularsingular}) is devoted to
regular-singular connections and to making explicit in the algebraic setting the link between
the bundles on the spaces $\tMgb{g}{n}$ and those on their proper replacements $\Mgrbt{g}{n}{r}{s}$.


\subsection{Simpson's geometricity conjecture}


A large class of quasi-geometric flat bundles is provided by Gauss-Manin connections.
Gauss-Manin vector bundles with flat connections on $X$ a smooth Deligne-Mumford stack are constructed
through the following recipe: choose $U\subset X$ Zariski open dense,
$f:P\ra U$ a fibration in smooth proper varieties whose fiber above $x\in U$ we denote $P_x$,
and $k\geq 0$ an integer ; then the vector bundle's fibers at $x\in U$ is the cohomology group $\mathrm{H}^k(P_x;\C)$,
and the connection is the so called Gauss-Manin connection given by the identification of the homotopy type of nearby fibers.
We then define geometric flat bundles to be the subquotients of Gauss-Manin flat bundles.

Simpson conjectured\footnote{Simpson's conjecture is only stated for curves, but for not necessarily semisimple local systems.} the following partial converse
to the quasi-geometricity of Gauss-Manin connections, which fits in the philosophy of the standard
conjectures in algebraic geometry.

\begin{conjecture}[variant of {\cite[3.]{simpsonTranscendentalAspectsRiemannHilbert1990}}]\label{conjectureSimpson}
  Let $Y$ be a complement to a normal crossing divisor in a smooth proper Deligne-Mumford stack.
  Then any semisimple quasi-geometric flat bundle on $Y$ is geometric.
\end{conjecture}

Hence as semisimple conformal blocks are always quasi-geometric, we can expect them to have geometric constructions.
And indeed, for the modular categories $\Cat_{\mathfrak{g},\zeta}$ constructed from simple Lie algebras,
it is known, though the works of Feigin-Schechtman-Varchenko \cite{feiginAlgebraicEquationsSatisfied1995},
Ramadas \cite{ramadasHarderNarasimhanTraceUnitarity2009}, Looijenga \cite{looijengaUnitaritySL2conformalBlocks2009},
Belkale-Mukhopadhyay \cite{belkaleConformalBlocksCohomology2014} and Belkale-Fakhruddin \cite{belkaleConformalBlocksGenus2023},
that the associated conformal blocks in genus $0$ on the configuration spaces $\Conf{n}{\C}/\C$
are geometric (see \cite[1.5]{godfardConstructionHodgeStructures2024} for a brief review of references).
It is not known whether this fact extends to any modular category,
and quasi-geometricity gives a motivation to study this question of geometricity for families of examples beyond the $\Cat_{\mathfrak{g},\zeta}$,
for example the Drinfeld centers of (generalized) Haagerup fusion categories \cite{grossmanDrinfeldCentersFusion2023}.

Another consequence of \Cref{conjectureSimpson} is that semisimple conformal blocks should support complex variations of Hodge structures.
The existence of such Hodge structures is proved by the author in \cite{godfardHodgeStructuresConformal2024}.

The conjecture, as stated, assumes the semisimplicity of the local system. In the work \cite{godfardSemisimplicityConformalBlocks2025},
we prove that conformal blocks are always semisimple, hence fit into the framework of the above conjecture.


\subsection{Outline of the proof}


The proof of quasi-geometricity is an adaptation of the following simple proof that, for $X$ smooth projective over a number field $\Q$,
any rigid semisimple\footnote{Semisimplicity is used to ensure that the orbit of the monodromy representation is closed.
In the context of modular or braided functors, the orbit of the associated modular or braided category will always be closed,
hence the absence of a semisimplicity assumption in our main result.}
flat bundle $(\cE,\nabla)$ on $X_\C$ can be defined over $X_K$ for some number field $K$.

\begin{proof}
  By finiteness, $(\cE,\nabla)$ is in fact defined over a $\Q$-sub-algebra $A\subset \C$ of finite type,
  which we can assume to be smooth. Then we have a family $(\cE_t,\nabla_t)$ of flat bundles
  over $X_\C$ parametrized by the smooth complex manifold $B=(\Spec A\otimes_\Q \C)^\mathrm{an}$.
  Rigidity of the monodromy representation associated to $(\cE,\nabla)$ implies that
  $(\cE_t,\nabla_t)$ is isomorphic to $(\cE,\nabla)$ for all $t\in B$. Now $B$ has a $K$-point $t_0$ for some number field $K$,
  and then $(\cE_{t_0},\nabla_{t_0})$ is defined over $X_K$ and isomorphic to $(\cE,\nabla)$.  
\end{proof}

When trying to adapt the above proof to the case of a braided/modular functor, a finiteness problem arises:
the functor is comprised of infinitely many flat bundles and hence we may be unable to find a sub-algebra $A\subset \C$
of definition that is of finite type over $\Q$. A solution to this problem is to work with truncations of the braided/modular functor,
comprised of only a finite number of flat bundles.
In fact, no data is lost when truncating and truncated braided/modular functors are deformation rigid (see \Cref{forgetsoverC}).
Using this, the above proof goes through and provides the result of \Cref{mainresult} for truncations of braided/modular functors.

The last step of the proof is to show that if a truncation of a braided/modular functor can be defined over a number field $K$,
then the whole braided/modular functor can be defined over the same $K$. This is achieved through results on the fullfaithfulness
of the truncation over number fields (see \Cref{subsectionsubfields}).


\subsection{Organization of the paper}


In \Cref{sectionbundles}, we introduce categories of modular/ribbon/braided fusion categories over $\C$, modular/genus $0$ modular/braided functors
over any characteristic $0$ field, and detail the relationship between the former and the latter when the field of definition is $\C$.
\Cref{twistedmodulispaces} is devoted to defining over $\Q$ the twisted moduli spaces of curves on which modular functors live.

\Cref{sectionBK} explains the Bakalov-Kirillov definition of modular functors, where regular-singular connections are considered instead
of connections on twisted compactifications. The definition is then explicitly related to the one of \Cref{sectionbundles}.
The definition of specialization maps and some proofs are deferred to \Cref{appendixregularsingular}.

\Cref{sectionmodular} is devoted to the statement and proof of the main result, quasi-geometricity of conformal blocks. On the way, truncations
of modular functors and Ocneanu rigidity are discussed.

The purpose of \Cref{appendixregularsingular} on regular-singular connections is threefold:
to explain some subtleties appearing when working on non-algebraically closed fields, to introduce and prove well-definiteness of the specialization
of a regular-singular connection to the normal bundle of a divisor, and to study the behavior of regularity and the specialization map
with respect to taking root stack on a divisor.


\subsection{Acknowledgements}


This paper forms part of the PhD thesis of the author.
The author thanks Julien Marché for his help in writing this paper.
The author is especially thankful to Pavel Etingof for suggesting to work on his conjecture with Varchenko and for helpful correspondences.
The author also thanks Jean-Benoît Bost and Aleksander Zakharov for helpful discussions.


\section{Bundles with flat connection associated to modular, ribbon and braided categories}\label{sectionbundles}

The subject of these notes is the bundles with connections over (twisted) moduli spaces of curves associated to braided fusion categories,
ribbon fusion categories and modular fusion categories. This section introduces the relevant objects.

\subsection{Modular/ribbon/braided fusion categories}

For definitions of monoidal categories, rigid monoidal categories, fusion category, braided monoidal categories,
twists on braided monoidal categories and modular categories, see \cite[2.1, 2.10, 4.1, 8.1, 8.10 and 8.13]{etingofTensorCategories2015}.
In these notes, the base field for such categories will always be $\C$.

\begin{notation}
  For $\Lambda$ a finite set, $\VL$ denotes the category $(\mathrm{Vect}_\C)^\Lambda$ of $\Lambda$-colored vector spaces,
  i.e. $\Lambda$-indexed sequences $(V_\mu)_{\mu\in\Lambda}$ of vector spaces.
  The object of $\VL$ corresponding to the vector space $\C$ colored by $\lambda\in\Lambda$ will be denoted $[\lambda]$.
  In terms of sequences, $[\lambda]$ is $(V_\mu)_{\mu\in\Lambda}$ with $V_\lambda=\C$ and $V_\mu=0$ for $\lambda\neq\mu$.
\end{notation}

\begin{definition}[{\cite[5.5]{godfardHodgeStructuresConformal2024}}]\label{definitionribbon}
  Let $\Rib$ be the groupoid whose:
  \begin{description}
    \item[(1)] objects are pairs $(\Lambda,\VL)$, where $\Lambda$ is some finite set, and $\VL$ is endowed with
    the structure of a ribbon fusion category;
    \item[(2)] morphisms from $(\Lambda_1,\VLone)$ to $(\Lambda_2,\VLtwo)$ are pairs $(f,\phi)$, with $f:\Lambda_1\ra\Lambda_2$ a bijection and 
    $\phi:\otimes_1\simeq f^*\otimes_2$ a natural isomorphism, such that they induce a monoidal isomorphism $\VLone\ra\VLtwo$
    compatible with braidings and twists. In other words, morphisms form $(\Lambda_1,\VLone)$ to $(\Lambda_2,\VLtwo)$ are ribbon isomorphisms that
    are induced by a bijection $f:\Lambda_1\ra\Lambda_2$ at the level of the $\C$-linear categories and are strict on tensor units.
  \end{description}
  One similarly defines the groupoid $\Braid$ of braided fusion categories and the full sub-groupoid $\Modcat\subset \Rib$
  of modular fusion categories.
\end{definition}

\begin{remark}
  A fusion structure on $\VL$ induces a unique involution $\lambda\mapsto\lambda^\dagger$ of $\Lambda$
  and the choice of a fixed point $0\in\Lambda$ of this involution, such that for each $\lambda$, $[\lambda]^*\simeq [\lambda^\dagger]$
  and $1\simeq [0]$.
  Such a finite set with involution and preferred fixed point will be called a set of colors.
\end{remark}

\subsection{Twisted moduli spaces of curves: definition over rational numbers}\label{twistedmodulispaces}

See \cite[2.4]{godfardHodgeStructuresConformal2024} for details on the twisted moduli spaces of curves over $\C$.
The aim of this section is just to briefly explain how these moduli spaces and the natural maps between them are in fact defined over $\Q$.
We reuse notation and terminology from \cite[2]{godfardHodgeStructuresConformal2024}.
By $\Mg{g}{n}$ we mean the moduli space of genus $g$ curves with $n$ disjoint sections seen as a stack over $\Q$,
and by $\oMg{g}{n}$ the Deligne-Mumford compactification, also over $\Q$. Fix $2$ integers $r,s\geq 1$.

Denote by $\bbV\otMgb{g}{n}$ the product over $\oMg{g}{n}$ of the total spaces of the $n$ tangent bundles $T_i$, $i=1,\dotsc,n$, to the sections
and of the anomaly bundle $H_{\mathcal{M}}$ (see \cite[2.42]{godfardHodgeStructuresConformal2024}).
For $\overline{\mathcal{C}}_{g,n}\ra \oMg{g}{n}$ the universal curve and $\sigma_i:\oMg{g}{n}\ra \overline{\mathcal{C}}_{g,n}$
the $i$-th section, $T_i$ is defined as $\sigma_i^*T\overline{\mathcal{C}}_{g,n}$.

Let $\tMgb{g}{n}$ be the complement in $\bbV\otMgb{g}{n}$
of the normal crossing divisor $D=D_\partial\cup D_1\cup\dotsb\cup D_n\cup D_H$, where $D_\partial$ is the locus of singular curves,
$D_i$ is the zero section of the tangent bundle at the $i$-th section, and $D_H$ is the zero section of $H_{\mathcal{M}}$.

Let $\bbV\Mgrbt{g}{n}{r}{s}$ be the stack obtained by taking root stack independently locally on each component of $D$,
with order $r$ on all components of $D_\partial\cup D_1\cup\dotsb\cup D_n$ and order $s$ on $D_H$.
In the notation of \cite[4.5]{chiodoStableTwistedCurves2008}, decomposing $D_\partial$ into components  $\bigcup_k D_{\partial,k}$, this means:
\begin{equation*}
  \bbV\Mgrbt{g}{n}{r}{s} = \bbV\otMgb{g}{n}\left[\sum_k \frac{D_{\partial,k}}{r}+\sum_i\frac{D_i}{r}+\frac{D_H}{s}\right].
\end{equation*}
See \cite[2.3, 4.5]{chiodoStableTwistedCurves2008} for more details on what is meant by independent root stack on a normal crossing divisor
with self-crossings in this context.
The twisted moduli space $\Mgrbt{g}{n}{r}{s}$ is then the reduced stack associated to the pullback of
the intersection $D_1\cap\dotsb\cap D_n\cap D_H$ of the zero sections along the root stack map $\bbV\Mgrbt{g}{n}{r}{s}\ra\bbV\otMgb{g}{n}$,
i.e. $\Mgrbt{g}{n}{r}{s}$ is $X^\mathrm{red}$ with $X$ as below:
\[\begin{tikzcd}
	X && {\bbV\Mgrbt{g}{n}{r}{s}} \\
	\\
	{D_1\cap\dotsb\cap D_n\cap D_H} && {\bbV\otMgb{g}{n}.}
	\arrow[from=1-1, to=1-3]
	\arrow[from=1-1, to=3-1]
	\arrow["\lrcorner"{anchor=center, pos=0.125}, draw=none, from=1-1, to=3-3]
	\arrow[from=1-3, to=3-3]
	\arrow[from=3-1, to=3-3]
\end{tikzcd}\]

The stack $\Mgrbt{g}{n}{r}{s}$, defined over $\Q$, coincides over $\C$ with the one similarly denoted in \cite[2.4]{godfardHodgeStructuresConformal2024}
(see \cite[4.5]{chiodoStableTwistedCurves2008}).

Removing the anomaly bundle $H_{\mathcal{M}}$ from the definition above yields a stack $\Mgrb{g}{n}{r}$ defined over $\Q$,
which, again, coincides over $\C$ with the one similarly denoted in \cite[2.4]{godfardHodgeStructuresConformal2024}.
Note that when $g=0$, the anomaly bundle is canonically trivial by definition and hence $\Mgrbt{0}{n}{r}{s}\simeq \Mgrb{0}{n}{r}\times\mathrm{B}\mu_s$ canonically.

The gluing and forgetful maps between the moduli spaces $\oMg{g}{n}$ induce, through the root stack construction above,
similar maps between the stacks $\Mgrbt{g}{n}{r}{s}$ and between the stacks $\Mgrb{g}{n}{r}$.
For more details on these maps, see \cite[2.28, 2.29]{godfardHodgeStructuresConformal2024} and \Cref{{definitiongeometricmodularfunctor}} below.

\begin{remark}
  As in \cite[2.26]{godfardHodgeStructuresConformal2024}, we use the conventions $\Mgrb{0}{2}{r}\simeq\Br{r}$ and $\Mgrbt{0}{2}{r}{s}\simeq\Br{r}\times\Br{s}$.
  See there for more details.
\end{remark}

\subsection{Modular functors}

We reproduce here the definition of geometric modular functor given in \cite[2.4]{godfardHodgeStructuresConformal2024},
but on any field of characteristic $0$, using the fact, explained in \Cref{twistedmodulispaces}, that the stacks $\Mgrbt{g}{n}{r}{s}$
and $\Mgrb{g}{n}{r}$ are defined over $\Q$. For $K$ a field of characteristic $0$, we will denote by $\Mgrbt{g}{n}{r}{s}_K$ and $\Mgrb{g}{n}{r}_K$
the corresponding stacks over $K$.

\begin{definition}[Modular Functor]\label{definitiongeometricmodularfunctor}
  Let $\Lambda$ be a set of colors, $r,s\geq 1$ integers and $K$ a field of characteristic $0$.
  Then a geometric modular functor over $K$ with level $(r,s)$ is the data,
  for each $g,n\geq 0$ with $(g,n)\neq (0,0),(0,1),(1,0)$ and $\ul\in\Lambda^n$,
  of a bundle with flat connection $(\Nu_g(\ul),\nabla)$ over $\Mgrbt{g}{n}{r}{s}_K$, together with some isomorphisms
  described below:
  \begin{description}
      \item[(G-sep)] For each gluing map
      \begin{equation*}
        q:\Mgrbt{g_1}{n_1+1}{r}{s}_K\times\Mgrbt{0}{2}{r}{s}_K\times\Mgrbt{g_2}{n_2+1}{r}{s}_K\lra \Mgrbt{g_1+g_2}{n_1+n_2}{r}{s}_K
      \end{equation*}
      and each $\ul$, an isomorphism preserving the connections:
      \begin{equation*}
        q^*\Nu_{g_1+g_2}(\lambda_1,\dotsc,\lambda_n)\simeq \bigoplus_\mu\Nu_{g_1}(\lambda_1,\dotsc,\lambda_{n_1},\mu)\otimes
        \Nu_0(\mu,\mu^\dagger)^\vee\otimes
        \Nu_{g_2}(\lambda_{n_1+1},\dotsc,\lambda_n,\mu^\dagger);
      \end{equation*}
      \item[(G-nonsep)] For each gluing map
      \begin{equation*}
        p:\Mgrbt{g-1}{n+2}{r}{s}_K\times\Mgrbt{0}{2}{r}{s}_K\lra \Mgrbt{g}{n}{r}{s}_K
      \end{equation*}
      and each $\ul$, an isomorphism preserving the connections:
      \begin{equation*}
        p^*\Nu_g(\lambda_1,\dotsc,\lambda_n)\simeq \bigoplus_\mu\Nu_{g-1}(\lambda_1,\dotsc,\lambda_n,\mu,\mu^\dagger)
        \otimes \Nu_0(\mu,\mu^\dagger)^\vee;
      \end{equation*}
      \item[(N)] For each forgetful map $f:\Mgrbt{g}{n+1}{r}{s}_K\ra\Mgrbt{g}{n}{r}{s}_K$, and each $\ul$, an isomorphism preserving the connections:
      \begin{equation*}
        f^*\Nu_g(\lambda_1,\dotsc,\lambda_n)\simeq \Nu_g(\lambda_1,\dotsc,\lambda_n,0)
      \end{equation*}
      and a canonically isomorphism $(\Nu_0(0,0),\nabla)\simeq (\cO,d)$ (trivial flat bundle);
      \item[(Perm)] For each $\ul\in\Lambda^n$ and permutation $\sigma\in S_n$, an isomorphism:
      \begin{equation*}
        \Nu_g(\lambda_1,\dotsc,\lambda_n)\simeq \sigma^*\Nu_g(\lambda_{\sigma(1)},\dotsc,\lambda_{\sigma(n)}).
      \end{equation*}
  \end{description}
  The isomorphisms of \textbf{(G-sep)}, \textbf{(G-nonsep)}, \textbf{(N)} and \textbf{(Perm)}
  are to be compatible with each other and repeated applications.
  Moreover, we ask for the gluing to be symmetric in the sense that
  for each gluing isomorphism above, the change of variable $\mu\mapsto\mu^\dagger$
  on the right-hand side has the same effect as permuting the summands and applying the \textbf{(Perm)} isomorphisms
  $\Nu_0(\mu,\mu^\dagger)\simeq\Nu_0(\mu^\dagger,\mu)$ induced by $\sigma\in S_2\setminus\{\id\}$.
  
  The functor is also assumed to verify the non-degeneracy axiom:
  \begin{description}
      \item[(nonD)] For each $\lambda$, $\Nu_0(\lambda,\lambda^\dagger)\neq 0$.
  \end{description}
\end{definition}

Sometimes, we will shorten $\Nu_0(\ul)$ to $\Nu(\ul)$.
Because $\Mgrbt{g}{n}{r}{s}_K$ is a $\mu_s$-gerbe, $\mu_s$ acts on the fibers of $\Nu_g(\ul)$.
Using the gluing axiom, one sees that this action is by scalars and independent of $g$ and $\ul$.
Hence to each modular functor $\Nu$ is associated a character $c$ of $\mu_s$ over $K$, called the \textbf{central charge} of $\Nu$.
When $K$ is given with an embedding into $\C$, we will identify $c$ with the complex number $c(e^{2\pi i/s})$.

\begin{remark}
  We use somewhat non-standard gluing maps involving $\Mgrbt{0}{2}{r}{s}$: the points are not directly glued together, but are first glued to the marked
  points of a twice marked sphere. This sphere is then contracted as it is an unstable component in the image, and hence adding it may seem pointless.
  However, we use it to have a more compact way to deal with the symmetry of the gluing. Indeed, even when $\mu=\mu^\dagger$,
  there are some examples for which the permutation isomorphism $\Nu_0(\mu,\mu)\simeq\Nu_0(\mu,\mu)$ induced by $\sigma\in S_2\setminus\{\id\}$
  is not the identity. See \cite[rmk. 2.7]{godfardHodgeStructuresConformal2024} for more details.
\end{remark}

\begin{definition}[Genus $0$ Modular Functor]
  Let $\Lambda$ be a set of colors, $r\geq 1$ an integer and $K$ a field of characteristic $0$.
  Then a geometric genus $0$ modular functor over $K$ with level $r$ is the data,
  for each $n\geq 2$ and $\ul\in\Lambda^n$,
  of a bundle with flat connection $(\Nu(\ul),\nabla)$ over $\Mgrb{0}{n}{r}_K$ together with isomorphisms
  as in \Cref{definitiongeometricmodularfunctor}, with $\Mgrbt{0}{n}{r}{s}_K$ replaced by $\Mgrb{0}{n}{r}_K$,
  satisfying all the same axioms\footnote{Note that \textbf{(G-nonsep)} is vacuous in genus $0$.}.
\end{definition}

\begin{remark}\label{remarklevel}
  The choice of level is not significant in the definition of (genus $0$) modular functor.
  Indeed, if $r'$ is a multiple of $r$ and $s'$ a multiple of $s$, pullback along the maps $\Mgrbt{g}{n}{r'}{s'}\ra\Mgrbt{g}{n}{r}{s}$
  produces a modular functor of level $(r',s')$ from a modular functor of level $(r,s)$.
  Similarly for genus $0$ modular functors.
\end{remark}

\subsection{Braided functors}

Fix a level $r\geq 1$. For $n\geq 1$, number the markings on curves in $\oMg{0}{n+1}$ from $0$ to $n$.
Denote by $\bbV\mysecondleftidx{^1}{\overline{\mathcal{M}}}{_{0,{n}}}$ the
total space of the tangent bundle to the $0$-th section in $\oMg{0}{n+1}$. As in \Cref{twistedmodulispaces}, consider the divisor with normal crossings
$D_\partial\cup D_0\subset\bbV\mysecondleftidx{^1}{\overline{\mathcal{M}}}{_{0,{n}}}$
where $D_\partial$ is the locus of singular curves and $D_0$ is the zero section of the bundle.

Define $\bbV\Mrpo{n}{r}$ as the space obtained by taking root stack of order $r$ independently on each component of $D_\partial\cup D_0$,
and $\Mrpo{n}{r}$ as the reduced stack associated to the pullback of $D_0\subset \bbV\mysecondleftidx{^1}{\overline{\mathcal{M}}}{_{0,{n}}}$ to $\bbV\Mrpo{n}{r}$.
We will use the convention that $\Mrpo{1}{r}=*$.

\begin{remark}\label{remarkconf}
  Note that $\bbV\mysecondleftidx{^1}{\overline{\mathcal{M}}}{_{0,{n}}}\setminus D_\partial\cup D_0$
  is the quotient $\Conf{n}{\A^1}/\mathbb{G}_a$ of the configuration space of $n$ ordered points in $\A^1$ by translations.
  The fundamental group of its analytification over $\C$ is then the pure braid group $\PB_n$ on $n$ strands.
  By Van-Kampen, the fundamental group of the analytification of $\Mrpo{n}{r}$ is the quotient of $\PB_n$
  by the $r$-th powers of Dehn twists. Hence $\Mrpo{n}{r}$ can be thought of as a suitable proper stack defined over $\Q$ replacing
  $\Conf{n}{\A^1}/\mathbb{G}_a$.
\end{remark}

The gluing and forgetful maps between the moduli spaces $\oMg{0}{n}$ induce similar maps between the stacks $\Mrpo{n}{r}$.
See \cite[2.3.3 and Figure 1.1]{godfardConstructionHodgeStructures2024} for more explanations.

\begin{definition}[Braided Functor, compare with {\cite[2.14, 2.15]{godfardHodgeStructuresConformal2024}}]\label{braidedfunctor}
  Let $\Lambda$ be a finite set, $r\geq 1$ an integer and $K$ a field of characteristic $0$.
  A geometric braided functor over $K$ with level $r$ is the data,
  for each $n\geq 1$, $\ul\in\Lambda^n$ and $\mu\in \Lambda$,
  of a bundle with flat connection $(\Nu(\mu;\ul),\nabla)$ over $\Mrpo{n}{r}_K$, together with some isomorphisms
  described below:
  \begin{description}
      \item[(G)] For each gluing map
      \begin{equation*}
        q:\Mrpo{n_1+1}{r}_K\times\Mrpo{n_2}{r}_K\lra \Mrpo{n}{r}_K
      \end{equation*}
      and each $\mu,\ul$, an isomorphism preserving the connections:
      \begin{equation*}
        q^*\Nu(\mu;\lambda_1,\dotsc,\lambda_n)\simeq \bigoplus_\nu\Nu(\mu;\lambda_{1},\dotsc,\lambda_{n_1},\nu)
        \otimes\Nu(\nu;\lambda_{n_1+1},\dotsc,\lambda_{n});
      \end{equation*}
      \item[(N)] For each forgetful map $f:\Mrpo{n+1}{r}_K\ra\Mrpo{n}{r}_K$, and each $\mu,\ul$, an isomorphism preserving the connections:
      \begin{equation*}
        f^*\Nu(\lambda_1,\dotsc,\lambda_n)\simeq \Nu(\lambda_1,\dotsc,\lambda_n,0)
      \end{equation*}
      and for each $\lambda\in\Lambda$ a canonically isomorphism $(\Nu(\lambda;\lambda),\nabla)\simeq (\cO,d)$ (trivial flat bundle);
      \item[(Dual)] For each $\lambda$ there exists a unique $\mu$ such that $\Nu(0;\lambda,\mu)\neq 0$.
      This $\mu$ will be denoted $\lambda^\dagger$;
      \item[(Perm)] For each $\mu,\ul\in\Lambda^n$ and permutation $\sigma\in S_n$, an isomorphism:
      \begin{equation*}
        \Nu(\mu;\lambda_1,\dotsc,\lambda_n)\simeq \sigma^*\Nu(\mu;\lambda_{\sigma(1)},\dotsc,\lambda_{\sigma(n)}).
      \end{equation*}
  \end{description}
  The isomorphisms of \textbf{(G)}, \textbf{(N)} and \textbf{(Perm)}
  are to be compatible with each other and repeated applications.
\end{definition}

\Cref{remarklevel} on the unimportance of the choice of the level also applies to geometric braided functors.

\begin{terminology}
  The generic term \textit{geometric functor} will refer to a geometric modular functor, a genus $0$ geometric modular functor,
  or a geometric braided functor.
\end{terminology}

\subsection{Relationship with modular/ribbon/braided fusion categories}

\begin{definition}
  Let $K$ be a field and $c$ a character of $\mu_s$ over $K$. We will denote by $\GMod_c(K)$ the groupoid of geometric modular functors
  defined over $K$ with central charge $c$, where we identify modular functors of different levels as in \Cref{remarklevel}.

  An isomorphism between $\Nu$ and $\Nu'$ in $\GMod_c(K)$ is a bijection $\phi:\Lambda\simeq\Lambda'$ preserving the involutions and $0$ together with
  a family of isomorphisms $\Nu_g(\ul)\simeq\Nu'_g(\phi(\ul))$ compatible with gluing, forgetful, normalization and permutation isomorphisms.

  We will use the notation $\GMod^0(K)$ for the groupoid of genus $0$ geometric modular functors over $K$, up to change of level,
  and the notation $\BraidFun(K)$ for the groupoid of geometric braided functors, up to change of level.
\end{definition}

One can associate to a geometric (genus $0$) modular functor over $\C$ a ribbon weakly fusion category and to a geometric braided functor over $\C$ a
braided weakly fusion category, see \cite[5.]{godfardHodgeStructuresConformal2024}.
Moreover, Etingof and Penneys recently proved the long-standing conjecture that braided weakly fusion categories are braided fusion categories
\cite{etingofRigidityNonnegligibleObjects2024}. Together with the work of Bakalov and Kirillov \cite[5.4.1, 6.7.13]{bakalovLecturesTensorCategories2000},
this implies that to a modular functor is associated a modular fusion category, to a genus $0$ modular functor is associated a ribbon fusion category,
and to a braided functor is associated a braided fusion category.
These can be put into the commutative diagram of groupoids below.
For more details, see \cite[5.]{godfardHodgeStructuresConformal2024}. 
\[\begin{tikzcd}
	{\bigsqcup_c\GMod_c(\C)} & {\GMod^0(\C)} & {\BraidFun(\C)} \\
	\Modcat & \Rib & {\Braid.}
	\arrow[from=1-1, to=1-2]
	\arrow["{\sqcup_c\mathrm{f}_c}"', from=1-1, to=2-1]
	\arrow[from=1-2, to=1-3]
	\arrow["{\mathrm{f}^0}", from=1-2, to=2-2]
	\arrow["{\mathrm{f}^b}", from=1-3, to=2-3]
	\arrow[from=2-1, to=2-2]
	\arrow[from=2-2, to=2-3]
\end{tikzcd}\]

The Reshetikhin-Turaev construction \cite{reshetikhinInvariants3manifoldsLink1991} provides a functor the other way:
given a modular category, it provides a $2+1$ TQFT and hence a modular functor. Here is a statement encapsulating what we will need from this construction
and its simpler genus $0$ variants.

\begin{theorem}[{\cite[5.9]{godfardHodgeStructuresConformal2024}, based on \cite[5.4.1, 6.7.13]{bakalovLecturesTensorCategories2000}}]\label{fullfaithfulnesses}
  The functors $\mathrm{f}^0$ and $\mathrm{f}^b$ are equivalences.
  The functors $\mathrm{f}_c$ are fully faithful, and every element of $\Modcat$
  is in the essential image of some $\mathrm{f}_c$ for a well chosen value\footnote{See \cite[5.7.10]{bakalovLecturesTensorCategories2000} for the exact values $c$ may take.}
  of $c$.
\end{theorem}

\begin{remark}
  Our definitions of braided and modular functors incorporate a level $(r,s)$ or $r$, whereas those of Bakalov-Kirillov do not.
  In fact a braided or modular functor in the sense of Bakalov-Kirillov always admits a level, i.e. the monodromies around boundary
  divisors are always of finite order, see \cite[2.17, 2.18]{godfardHodgeStructuresConformal2024}.
\end{remark}


\section{Flat connections on root stack versus flat connections with regular singularities}\label{sectionBK}

The aim of this section is to compare the definition of geometric functors given here to that of Bakalov-Kirillov
\cite[6.4.1]{bakalovLecturesTensorCategories2000},
and explain how,
over any field $K$ of characteristic $0$, to obtain a modular functor in the sense of Bakalov-Kirillov from
one defined on twisted moduli spaces. This explain why \Cref{mainresult} also applies to modular functors in the sense of Bakalov-Kirillov.


In this section, we will use connection with regular singularities and specialization maps.
These notions are detailed in \Cref{appendixregularsingular}. For $D$ a normal crossing divisor in a smooth proper Deligne-Mumford
stack $X$ over a field of characteristic $0$,
we will denote by $\RS{X}{D}$ the category of bundles with flat connections on $X\setminus D$ whose singularities are regular along $D$.
When the choice of $(X,D)$ compactifying $X\setminus D$ is clear, we will abbreviate to $\RSo{X\setminus D}$.

Let $\bbV\otMgb{g}{n}$ be as in \Cref{twistedmodulispaces} and let $\tMgb{g}{n}$ be the complement of
$D=D_\partial\cup D_1\cup\dotsb\cup D_n\cup D_H$ in $\bbV\otMgb{g}{n}$. Gluing maps between the moduli spaces $\oMg{g}{n}$
and the description of the tangent bundles to components of $D_\partial$ induce specialization maps (see \cite[6.3.15]{bakalovLecturesTensorCategories2000}
and \Cref{subsectionspecialization})\footnote{These are compositions of specialization maps with pullbacks.}
\begin{align*}
  \mathrm{Sp}(q)&:\RSo{{\tMgb{g_1+g_2,K}{n_1+n_2}}}\lra \RSo{{\tMgb{g_1,K}{n_1+1}}\times\tMgb{0,K}{2}\times{\tMgb{g_2,K}{n_2+1}}},\\
  \mathrm{Sp}(p)&: \RSo{{\tMgb{g,K}{n}}}\lra\RSo{{\tMgb{g-1,K}{n+2}}\times\tMgb{0,K}{2}}.
\end{align*}

Note that $\tMgb{0,K}{2}$ is isomorphic to $\mathbb{G}_m^{2}$.
One can use these to replace pullbacks in the definition of modular functor.

\begin{remark}\label{remarkequivariance}
  Note that $\tMgb{g}{n}$ is a principle $\mathbb{G}_m^{n+1}$-bundle over $\Mg{g}{n}$.
  Hence we may consider flat bundles on $\tMgb{g}{n}$ which are $\mathbb{G}_m^{n+1}$-equivariant.
  Then, the specialization maps above preserve equivariance, as mentioned in \Cref{remarkequivarianceappendix}. For example, $\mathrm{Sp}(p)$ sends $\mathbb{G}_m^{n+1}$-equivariant flat bundles
  to $\mathbb{G}_m^{n+3}\times\mathbb{G}_m^{3}$-equivariant flat bundles ($\tMgb{0,K}{2}\simeq \mathbb{G}_m^{2}$ is a principle
  $\mathbb{G}_m^{3}$-bundle over $\Mg{0}{2}\simeq \mathrm{B}\mathbb{G}_m$).
  For simplicity, we will simply write \emph{equivariant} to mean $\mathbb{G}_m^k$-equivariant when the exponent $k$ and the $\mathbb{G}_m^k$-actions are clear.
\end{remark}

\begin{definition}[{\cite[6.4.1]{bakalovLecturesTensorCategories2000}}]\label{definitiongeometricmodularfunctorBK}
  Let $\Lambda$ be a set of colors and $K$ a field of characteristic $0$.
  Then a geometric modular functor over $K$ in the sense of Bakalov-Kirillov is the data,
  for each $(g,n)$ stable and $\ul\in\Lambda^n$,
  of an equivariant bundle with regular-singular flat connection $(\Nu_g(\ul),\nabla)$ over $\tMgb{g,K}{n}$, together with some isomorphisms
  described below:
  \begin{description}
      \item[(G-sep)] For each gluing specialization map
      \begin{equation*}
        \mathrm{Sp}(q):\RSo{{\tMgb{g_1+g_2,K}{n_1+n_2}}}\lra \RSo{{\tMgb{g_1,K}{n_1+1}}\times\tMgb{0,K}{2}\times{\tMgb{g_2,K}{n_2+1}}}
      \end{equation*}
      and each $\ul$, an equivariant isomorphism preserving the connections:
      \begin{equation*}
        \mathrm{Sp}(q)\left(\Nu_{g_1+g_2}(\lambda_1,\dotsc,\lambda_n)\right)\simeq \bigoplus_\mu\Nu_{g_1}(\lambda_1,\dotsc,\lambda_{n_1},\mu)\otimes
        \Nu_0(\mu,\mu^\dagger)^\vee\otimes
        \Nu_{g_2}(\lambda_{n_1+1},\dotsc,\lambda_n,\mu^\dagger);
      \end{equation*}
      \item[(G-nonsep)] For each gluing specialization map
      \begin{equation*}
        \mathrm{Sp}(p): \RSo{{\tMgb{g,K}{n}}}\lra\RSo{{\tMgb{g-1,K}{n+2}}\times\tMgb{0,K}{2}}
      \end{equation*}
      and each $\ul$, an equivariant isomorphism preserving the connections:
      \begin{equation*}
        \mathrm{Sp}(p)\left(\Nu_g(\lambda_1,\dotsc,\lambda_n)\right)\simeq \bigoplus_\mu\Nu_{g-1}(\lambda_1,\dotsc,\lambda_n,\mu,\mu^\dagger)
        \otimes \Nu_0(\mu,\mu^\dagger)^\vee;
      \end{equation*}
      \item[(N)] For each forgetful map $f:\tMgb{g,K}{n+1}\ra\tMgb{g,K}{n}$, and each $\ul$, an equivariant isomorphism preserving the connections:
      \begin{equation*}
        f^*\Nu_g(\lambda_1,\dotsc,\lambda_n)\simeq \Nu_g(\lambda_1,\dotsc,\lambda_n,0)
      \end{equation*}
      and a canonical equivariant isomorphism $(\Nu_0(0,0),\nabla)\simeq (\cO,d)$ (trivial flat bundle);
      \item[(Perm)] For each $\ul\in\Lambda^n$ and permutation $\sigma\in S_n$, an equivariant isomorphism:
      \begin{equation*}
        \Nu_g(\lambda_1,\dotsc,\lambda_n)\simeq \sigma^*\Nu_g(\lambda_{\sigma(1)},\dotsc,\lambda_{\sigma(n)}).
      \end{equation*}
    \end{description}
    The isomorphisms of \textbf{(G-sep)}, \textbf{(G-nonsep)}, \textbf{(N)} and \textbf{(Perm)}
    are to be compatible with each other and repeated applications.
    Moreover, we ask for the gluing to be symmetric in the sense that
    for each gluing isomorphism above, the change of variable $\mu\mapsto\mu^\dagger$
    on the right-hand side has the same effect as permuting the summands and applying the \textbf{(Perm)} isomorphisms
    $\sigma^*\Nu_0(\mu,\mu^\dagger)\simeq\Nu_0(\mu^\dagger,\mu)$ induced by $\sigma\in S_2\setminus\{\id\}$.
    
    The functor is also assumed to verify the non-degeneracy axiom:
    \begin{description}
        \item[(nonD)] For each $\lambda$, $\Nu_0(\lambda,\lambda^\dagger)\neq 0$.
    \end{description}
\end{definition}

Let us first detail the relationship between the moduli spaces at play in the $2$ definitions of modular functor.
The projection $\bbV\otMgb{g}{n}\ra D_1\cap\dotsb\cap D_n\cap D_H$ induces a projection $p:\bbV\Mgrbt{g}{n}{r}{s}\ra\Mgrbt{g}{n}{r}{s}$,
and we get a diagram:
\[\begin{tikzcd}
	{\Mgrbt{g}{n}{r}{s}} & {\bbV\Mgrbt{g}{n}{r}{s}} & {\tMgb{g}{n}} & {\bbV\otMgb{g}{n}}
	\arrow["i"', curve={height=12pt}, from=1-1, to=1-2]
	\arrow["p", curve={height=-12pt}, tail reversed, no head, from=1-1, to=1-2]
	\arrow["{j'}"', hook', from=1-3, to=1-2]
	\arrow["j", hook, from=1-3, to=1-4]
\end{tikzcd}\]
where $j$ and $j'$ are open embeddings of complements to normal crossing divisors.


\begin{proposition}\label{twistedtoBK}
  Let $K$ be a field of characteristic $0$ and $\Nu$ be a geometric modular functor over $K$ with level $(r,s)$ and set of colors $\Lambda$,
  as in \Cref{definitiongeometricmodularfunctor}.
  For each $n$ and $\ul\in\Lambda^n$, denote by $\Nu^{BK}_g(\ul)$ the flat bundle $j'^*p^*\Nu_g(\ul)$ on $\tMgb{g}{n}$,
  with $j'$ and $p$ as above.
  Then the family $(\Nu^{BK}_g(\ul))_{g,\ul}$ is part of a modular functor over $K$ in the sense of Bakalov-Kirillov (\ref{definitiongeometricmodularfunctorBK}),
  whose forgetful, permutation and gluing isomorphisms are induced by those of $\Nu$.
  Moreover, the map $\Nu\mapsto \Nu^{BK}$ is compatible with field extensions $K\hookrightarrow L$.

  Appropriate variations of this statement hold for genus $0$ geometric modular functors and geometric braided functors.
\end{proposition}

\begin{remark}\label{remarkBKgenuszero}
  Removing $H_{\mathcal{M}}$ from the definition of $\tMgb{0}{n}$, one gets a space $\Mgb{0}{n}$.
  Then, in the Bakalov-Kirillov definition of genus $0$ modular functor, one replaces $\tMgb{0}{n}$ by $\Mgb{0}{n}$.
  Note that as $g=0$, $\tMgb{0}{n}=\Mgb{0}{n}\times\mathbb{G}_m$. As for the corresponding Bakalov-Kirillov definition of braided functors,
  one works with the spaces $\Conf{n}{\A^1}/\mathbb{G}_a$ (see \Cref{remarkconf}).
\end{remark}

The aim of the rest of this section is to explain the proof of this proposition. The main arguments are separated into the three Lemmas
below, whose proofs are delayed to \Cref{appendixregularsingular}.

The first step is to verify that $\Nu^{BK}_g(\ul)$ has regular singularities with respect to the embedding
$\tMgb{g}{n}\subset\bbV\otMgb{g}{n}$. By definition, it is regular with respect to the embedding
$\tMgb{g}{n}\subset\bbV\Mgrbt{g}{n}{r}{s}$. Hence regularity is provided by the following Lemma.

\begin{lemma}\label{regularrootstack}
  Let $X$ be a geometrically connected smooth separated Deligne-Mumford stack over a field $K$ of characteristic $0$.
  Let $D=D_1\cup\dotsb\cup D_n\subset X$ be a normal crossing divisor with each $D_i$ geometrically connected, possibly with normal self-crossings,
  and $r_i$, $i=1,\dotsc,n$ positive integers.
  Let $X':=X[\sum_i\frac{D_i}{r_i}]$ be the root stack, $D'=\bigcup_i\frac{D_i}{r_i}$ the divisor and $f:X'\ra X$ the projection.
  Then a connection on a bundle over $X\setminus D=X'\setminus D'$ is regular for $(X,D)$ if and only if it is regular for $(X',D')$.
  Moreover, for $(\cE,\nabla)$ regular, if $\overline{\cE}$ is an extension to $X$,
  then $f^*\overline{\cE}$ is an extension to $X'$, and if $\overline{\cE}'$ is an extension to $X'$,
  then $f_*\overline{\cE}'$ is an extension to $X$.
\end{lemma}

Vacuum and permutation isomorphisms of $\Nu$ clearly induce forgetful and permutation isomorphisms for $\Nu^{BK}$.
Let us turn to gluing isomorphisms. As a first step, the following Lemma shows that in the definition
of the gluing isomorphisms we may replace the pairs $\tMgb{g}{n}\subset\bbV\otMgb{g}{n}$
by the pairs $\tMgb{g}{n}\subset\bbV\Mgrbt{g}{n}{r}{s}$.

\begin{lemma}\label{specializationrootstack}
  Let $(X,D)$ and $(X',D')$ be as in \Cref{regularrootstack}. Let $\hat{D}_1$ be the normalization\footnote{As $D_1$ may have self-crossings, it is not necessarily smooth.}
  of $D_1$ and for $i\geq 1$, $\hat{D}_{1i}$ the pullback of $D_i\pitchfork D_1$ to $\hat{D}_1$.
  Similarly define $\hat{D}_1'$ and $\hat{D}_{1i}'$ for the pair $(X',D')$.
  Then specialization commutes with root stack, in the sense that the following diagram commutes:
  \[\begin{tikzcd}
    {\RS{X}{D}} && {\RS{N\hat{D}_1}{\hat{D}_1\cup N\hat{D}_{1\mid\bigcup_i\hat{D}_{1i}}}} \\
    {\RS{X'}{D'}} && {\RS{N\hat{D}_1'}{\hat{D}_1'\cup N\hat{D}'_{1\mid\bigcup_i\hat{D}_{1i}'}}}
    \arrow["{\text{specialization}}", from=1-1, to=1-3]
    \arrow[tail reversed, from=1-1, to=2-1]
    \arrow[tail reversed, from=1-3, to=2-3]
    \arrow["{\text{specialization}}", from=2-1, to=2-3]
  \end{tikzcd}\]
  where the vertical equivalences of categories are provided by \Cref{regularrootstack}.
\end{lemma}

See \Cref{appendixregularsingular} for notations on specialization. The gluing specialization isomorphisms for $\Nu^{BK}$ are then provided by the following
third Lemma, which is a direct consequence of the definition of the specialization map, see \Cref{subsectionspecialization}.
The compatibilities between different forgetful, permutation and gluing isomorphisms for $\Nu^{BK}$ then follow
from those of $\Nu$. This concludes the proof of \Cref{twistedtoBK}.

\begin{lemma}\label{pullbackandspecialization}
  Let $(X,D=D_1\cup\dotsb\cup D_n)$ be as in \Cref{regularrootstack}, and $\hat{D}_1$, $\hat{D}_{1i}$ as in \Cref{specializationrootstack}.
  Then the specialization to $N\hat{D}_1$ of a connection that extends to $X\setminus D_2\cup\dotsb\cup D_n$ is just its pullback.
  More precisely, the following diagram commutes:
  \[\begin{tikzcd}
    {\RS{X}{\bigcup_{i\geq 2}D_i}} & {\RS{\hat{D}_1}{\bigcup_{i\geq 2}\hat{D}_{1i}}} & {\RS{N\hat{D}_1}{N\hat{D}_{1\mid \bigcup_{i\geq 2}\hat{D}_{1i}}}} \\
    {\RS{X}{\bigcup_{i\geq 1}D_i}} && {\RS{N\hat{D}_1}{\hat{D}_1\cup N\hat{D}_{1\mid\bigcup_{i\geq 2}\hat{D}_{1i}}}.}
    \arrow["{\text{pullback}}", from=1-1, to=1-2]
    \arrow["{\text{restriction}}"', from=1-1, to=2-1]
    \arrow["{\text{pullback}}", from=1-2, to=1-3]
    \arrow["{\text{restriction}}", from=1-3, to=2-3]
    \arrow["{\text{specialization}}", from=2-1, to=2-3]
  \end{tikzcd}\]
\end{lemma}


\section{Proof that geometric functors are defined over number fields}\label{sectionmodular}


\subsection{Statement of the main result}

\begin{theorem}\label{mainresult}
  Let $\Nu$ be a geometric modular functor over $\C$, and let $(r,s)$ be a level for it.
  Then there exists a number field $K$
  such that $\Nu$ can be defined over $K$ in the sense of \Cref{definitiongeometricmodularfunctor},
  i.e. each bundle with flat connection and gluing, forgetful and permutation isomorphism defining $\Nu$
  can be defined over the twisted moduli spaces of curves $\Mgrbt{g}{n}{r}{s}_K$ over the field $K$.

  The same results hold for genus $0$ geometric modular functors and geometric braided functors.
\end{theorem}

Compatibility with field extensions in \Cref{twistedtoBK} and \Cref{mainresult} then imply the following.

\begin{corollary}\label{mainresultBK}
  Let $\Cat$ be a modular category. Then there exists a number field $K$ such that
  the corresponding geometric modular functor in the sense of Bakalov-Kirillov can be defined over $K$ (see \Cref{definitiongeometricmodularfunctorBK}),
  i.e. each bundle with flat connection, gluing specialization isomorphism, and forgetful and permutation isomorphism defining it
  can be defined over the moduli spaces $\tMgb{g,K}{n}$ over the field $K$.

  The same results hold for genus $0$ geometric modular functors coming from ribbon fusion categories and geometric braided functors coming from
  braided fusion categories.
\end{corollary}


\subsection{Truncated geometric functors}

\begin{definition}
  Let $\chi\leq -3$ be an integer.
  A $\chi$-truncated geometric modular functor with level $(r,s)$ and set of colors $(\Lambda,\dagger)$
  is the data, for each $g,n\geq 0$ \textbf{such that} $\bm{2-2g-n\geq \chi}$ and $\ul\in \Lambda^n$, of a bundle with flat connection
  $(\Nu_g(\ul),\nabla)\ra \Mgrbt{g}{n}{r}{s}$, together with the data of relevant gluing, forgetful and permutation isomorphisms
  satisfying the restriction of axioms of a modular functor to $2-2g-n\geq \chi$.

  One similarly defines $\chi$-truncated genus $0$ geometric modular functors and $\chi$-truncated geometric braided functors.
\end{definition}

Note that a $\chi$-truncated geometric functor comprises of finite data, i.e. of finitely many bundles and
gluing/forgetful/permutation isomorphisms.

\subsection{Full-faithfulness of forgetful functors over subfields of \texorpdfstring{$\C$}{complex numbers}}\label{subsectionsubfields}

\begin{notation}
  Let $K$ be a field of characteristic $0$ and $c$ be a central charge.
  We will denote by $\GMod_c^{\geq \chi}(K)$ the groupoid of $\chi$-truncated geometric modular functors
  defined over $K$ and with central charge $c$, by $\GMod^{0,\geq \chi}(K)$ the groupoid of $\chi$-truncated genus $0$ geometric modular functors over $K$,
  and by $\BraidFun^{\geq \chi}$ the groupoid of $\chi$-truncated geometric braided functors over $K$.
\end{notation}

\begin{theorem}[{\cite[5.11]{godfardHodgeStructuresConformal2024}}]\label{forgetsoverC}
  Fix $c\in\C^\times$. Then for any $\chi'\leq \chi\leq -3$, the following forgetful functors are
  well-defined and equivalences:
  \begin{equation*}
    \GMod_c(\C)\lra \GMod_c^{\geq \chi'}(\C)\lra \GMod_c^{\geq \chi}(\C)\lra \Modcat.
  \end{equation*}
  The same result holds for genus $0$ geometric modular functors and geometric braided functors, replacing $\Modcat$ by $\Rib$ and $\Braid$ respectively.
\end{theorem}

\begin{remark}
  Note that $\chi\leq -3$ is required for $\GMod_c^{\geq \chi}(\C)\lra \Rib$ to be defined.
\end{remark}

\begin{corollary}\label{forgetsoverK}
  Fix $K$ a subfield of $\C$ and $c$ a central charge.
  Then for any $\chi'\leq \chi\leq -3$, the following forgetful functors are fully faithful:
  \begin{equation*}
    \GMod_c(K)\lra \GMod_c^{\geq \chi'}(K)\lra \GMod_c^{\geq \chi}(K).
  \end{equation*}
  Again, the same holds for genus $0$ geometric modular functors and geometric braided functors.
\end{corollary}
\begin{proof}
  Faithfulness is a consequence of the same result over $\C$, which is contained in \Cref{forgetsoverC}.
  Let us prove fullness of $\GMod_c(K)\ra \GMod_c^{\geq \chi}(K)$.
  The other cases are similar.
  Let $\Nu^1$ and $\Nu^2$ be $2$ objects of $\GMod_c(K)$ and let $\Nu^{1,\geq\chi}$ and $\Nu^{2,\geq\chi}$
  denote their images in $\GMod_c^{\geq \chi}(K)$. Choose a morphism $f:\Nu^{1,\geq\chi}\ra\Nu^{2,\geq\chi}$ in $\GMod_c^{\geq \chi}(K)$.
  By \Cref{forgetsoverC}, $f$ lifts to a morphisms $\tilde{f}:\Nu^1\ra\Nu^2$ in $\GMod_c(\C)$. We want to show that it is
  in $\GMod_c(K)$. The group $\Aut{\C/K}$ acts on $\GMod_c(\C)$ and fixes $\Nu^1$ and $\Nu^2$, as they are defined over $K$.
  For any $\sigma$ in $\Aut{\C/K}$, $\sigma(\tilde{f})$ is a morphism $\Nu^1\ra\Nu^2$ lifting $f$.
  By faithfulness in \Cref{forgetsoverC}, one must have $\sigma(\tilde{f})=\tilde{f}$ for all $\sigma$ in $\Aut{\C/K}$.
  Now, $\tilde{f}$ can be seen as a family of isomorphisms over $\C$ between sheaves of $K$-vector spaces and hence is defined
  by some coefficients in $\C$. As the fixed points in $\C$ of $\Aut{\C/K}$ are exactly $K$, all these coefficients must be in $K$.
  Hence $\tilde{f}$ is defined over $K$.
\end{proof}


\subsection{Deformations of modular functors}

The main ingredient in the proof of our main result (\Cref{mainresult}) is Ocneanu rigidity, which we now state, and a textbook account of which is available in
\cite[chp. 9.1]{etingofTensorCategories2015}.

\begin{theorem}[{Ocneanu rigidity \cite[2.28]{etingofFusionCategories2005}}]\label{Ocneanurigidity}
  A fusion category does not admit nontrivial infinitesimal deformations (i.e. its associator does not).
  In particular, the number of such categories up to equivalence with a given Grothendieck ring is finite.
\end{theorem}

As a corollary of \Cref{Ocneanurigidity} we get.

\begin{corollary}\label{corollarydeformations}
  Let $C$ be a ribbon or braided category over $\C$.
  Then for any continuous family of ribbon or braided fusion categories $(C_t)_{t\in X}$, with $C_0=C$, $X$ path-connected,
  and where only the associators vary, $C_t$ is isomorphic to $C$ for all $t$.
\end{corollary}
\begin{proof}
  By \cite[7.7]{godfardHodgeStructuresConformal2024}
  (which is a corollary of \Cref{Ocneanurigidity}), the isomorphism class of $C_t$ is constant
  along arcs in $X$. As $X$ is path-connected, the isomorphism class of $C_t$ is independent of $t$.
\end{proof}


\subsection{Proof of the main result}

In this section, $\Nu$ is a geometric modular/genus $0$ modular/braided functor over $\C$ with
associated modular/ribbon/braided category $\CN$.

\textbf{Step 1:} The truncations of $\Nu$ are defined over number fields.

For $\chi\leq -3$, denote by $\Nu^{\geq\chi}$ the $\chi$-truncation of $\Nu$. As a first step, we fix $\chi$
and prove the theorem for $\Nu^{\geq\chi}$. A truncated geometric functor comprises of finitely many bundles with connection
and finitely many gluing, forgetful and permutation isomorphisms. Hence $\Nu^{\geq\chi}$ can be defined over a sub-algebra $A\subset\C$
of finite type over $\Q$. The algebra $A$ is integral of characteristic $0$ and so by generic smoothness there exists $f\in A\setminus\{0\}$
such that the localization $A_f$ is smooth over $\Q$. So we may assume that $A$ is smooth over $\Q$.

Let $X:=\Spec(A)$, $X^\mathrm{an}:=(X\times_\Q\C)^\mathrm{an}$ be the analytification and $t_0\in X^\mathrm{an}$
the point associated to the generic point of $X$.
Then we have a family of truncated geometric functors $(\Nu^{\geq\chi}_t)_{t\in X^\mathrm{an}}$
and an associated family of ribbon/braided weakly fusion categories $(C_t)_{t\in X^\mathrm{an}}$.
By \cite{etingofRigidityNonnegligibleObjects2024}, braided weakly fusion is equivalent to braided fusion.
Hence $(C_t)_{t\in X^\mathrm{an}}$ is a family of ribbon/braided fusion categories.
Note that by the construction of the functors $\GMod_c^{\geq \chi}(\C)\ra \Modcat$, $\GMod^{0,\geq \chi}(\C)\ra \Rib$ and $\BraidFun^{\geq \chi}(\C)\ra \Braid$
(see \cite[section 5]{godfardHodgeStructuresConformal2024}), only the associators vary.
As $X^\mathrm{an}$ is path-connected, by \Cref{corollarydeformations},
$C_t$ is isomorphic to $C_{t_0}$ for all $t\in X^\mathrm{an}$. By fullness in \Cref{forgetsoverC},
$\Nu^{\geq\chi}_t$ is isomorphic to $\Nu^{\geq\chi}_{t_0}=\Nu^{\geq\chi}$ for all $t$.

As $X$ is of finite type over $\Q$, it admits a $K$-point $t_1$ for some number field $K$.
Now $\Nu^{\geq\chi}_{t_1}$ is isomorphic to $\Nu^{\geq\chi}$ and definable over $K$.
So $\Nu^{\geq\chi}$ can be defined over $K$.

\textbf{Step 2:} Gluing $K$-forms of truncations of $\Nu$.

The data of a geometric functor $\Nu_K$ over a number field $K$ is equivalent to the data, for each $\chi\leq -3$,
of a $\chi$-truncated geometric functor $\Nu^{\geq\chi}_K$ over $K$, together with that of, for each $\chi\leq -3$,
an isomorphism $\phi_\chi$ between $\Nu^{\geq\chi}_K$ and the $\chi$-truncation of $\Nu^{\geq\chi-1}_K$.

Let $K$ be a number field over which $\Nu^{\geq -3}$ can be defined.
Let us construct such data over $K$ for $\Nu$. First, choose a $K$ form $\Nu^{\geq -3}_K$ of $\Nu^{\geq -3}$.
Then, define $\Nu^{\geq\chi}_K$ and $\phi_\chi$ by induction as follows. Assume that $\Nu^{\geq\chi}_K$ is defined.
Then, by the first step, $\Nu^{\geq\chi-1}$ can be defined over a number field $L$, which we can assume
to contain $K$. Choose an $L$-form $\widetilde{\Nu}^{\geq\chi-1}_L$ of $\Nu^{\geq\chi-1}$.
Applying \Cref{isomorphismdescent} below,
after replacing $L$ by a finite extension, we may assume that the $\chi$-truncation of $\widetilde{\Nu}^{\geq\chi-1}_L$
is isomorphic to the pullback of $\Nu^{\geq\chi}_K$ to $L$.
Then, by \Cref{Galoisdescent} below, $\widetilde{\Nu}^{\geq\chi-1}_L$ has a $K$-form
whose $\chi$-truncation is isomorphic to $\Nu^{\geq\chi}_K$. Choose $\Nu^{\geq\chi-1}_K$ to be that $K$-form
and $\phi_\chi$ to be an isomorphism.

From this data, as mentioned above, we get a geometric functor $\Nu_K$ over $K$. As its $(-3)$-truncation
is a $K$-form of the $(-3)$-truncation of $\Nu$, by fullness in \Cref{forgetsoverC}, $\Nu_K$ is a $K$-form of $\Nu$.
This completes the proof of \Cref{mainresult}.

\begin{lemma}\label{isomorphismdescent}
  Let $\chi\leq -3$. If two $\chi$-truncated geometric functors $\Nu_1$ and $\Nu_2$ defined over a number field $L$
  are isomorphic over $\C$, then there exists a finite extension $L'$ of $L$ on which they are isomorphic.
\end{lemma}
\begin{proof}
  The proof is very similar to that of Step 1 above. Let $\phi$ be the isomorphism over $\C$ between $\Nu_1$ and $\Nu_2$.
  Then $\phi$ and $\phi^{-1}$ are defined over a sub-$L$-algebra $B\subset \C$ of finite type.
  Localizing, we may choose $B$ to ensure that $\phi\circ\phi^{-1}=\id$ and $\phi^{-1}\circ\phi=\id$ over $B$.
  As $B$ is of finite type, there exists a surjection $B\twoheadrightarrow L'$ with $L'$ a finite extension of $L$.
  Applying this map to the coefficients of $\phi$, we get an isomorphism between $\Nu_1$ and $\Nu_2$ over $L'$.
\end{proof}

\begin{lemma}\label{Galoisdescent}
  Let $K$ be a number field and $L$ a Galois extension of $K$. Fix $\chi'\leq\chi\leq -3$ and let $\Nu$ be a
  $\chi'$-truncated geometric functor over $L$ such that it's $\chi$-truncation $\Nu^{\geq\chi}$ is the pullback of a
  $\chi$-truncated geometric functor $\Nu'$ defined over $K$. Then $\Nu$ can be defined over $K$,
  and it can be done so that its $\chi$-truncation is isomorphic to $\Nu'$.
\end{lemma}
\begin{proof}
  Let us denote by $\iota:K\hookrightarrow L$ the extension. Then the isomorphism $\Nu^{\geq\chi}\simeq\iota^*\Nu'$
  provides $\iota$-Galois descent data for $\Nu^{\geq\chi}$, i.e. the datum for each $\sigma\in\Gal(L/K)$,
  of an isomorphism $\phi_\sigma:\Nu^{\geq\chi}\ra\sigma^*\Nu^{\geq\chi}$, such that $\forall \sigma,\tau$,
  $\sigma^*(\phi_\tau)\circ\phi_\sigma=\phi_{\tau\sigma}$. By full-faithfulness of $\chi$-truncation over $L$
  (\Cref{forgetsoverK}), this gives $\iota$-Galois descent data for $\Nu$. Now by Galois descent applied
  to each module with connection $(\Nu_g(\ul),\nabla)$ and each gluing/forgetful/permutation isomorphism,
  $\Nu$ can be defined over $K$, say $\Nu\simeq\iota^*\widetilde{\Nu}$ with $\widetilde{\Nu}$ over $K$.
  Moreover, as the descent data came from $\Nu'$,
  the $\chi$-truncation of $\widetilde{\Nu}$ is isomorphic to $\Nu'$.
\end{proof}


\appendix


\section{Regular singular connections and specialization on root stacks in characteristic \texorpdfstring{$0$}{0}}\label{appendixregularsingular}

The purpose of this Appendix on regular-singular connections is threefold:
to explain some subtleties appearing when working on non-algebraically closed fields, to introduce and prove well-definiteness of the specialization
of a regular-singular connection to the normal bundle of a divisor, and to study the behavior of regularity and the specialization map
with respect to taking root stack on a divisor.

In this section, $X$ is a geometrically connected smooth separated Deligne-Mumford stack over a field $K$ of characteristic $0$,
and $D\subset X$ is a normal crossing divisor with irreducible components $D_1,\dotsc,D_n$, possibly with self-crossings,
that we each assume to be geometrically connected.


\subsection{Regular singular connections}

\begin{definition}\label{regularsingular}
  A bundle with flat connection $(\cE,\nabla)$ on $X\setminus D$ is said to be regular with respect to $(X,D)$
  if there exists a bundle $\overline{\cE}$ on $X$ extending $\cE$ such that $\nabla$ maps
  $\overline{\cE}$ into $\ocE\otimes\Omega_X^1(\log D)$. We call $\ocE$ an extension of $(\cE,\nabla)$.
\end{definition}

\begin{notation}
  The category of bundles with flat connection on $X\setminus D$ regular with respect to $(X,D)$ will be denoted,
  as in \Cref{sectionBK}, by $\RS{X}{D}$.
\end{notation}

Let us study locally the possible extensions $\overline{\cE}$ of $(\cE,\nabla)$. Notice that by Hartogs' Lemma,
such an extension is determined by its restriction to $X\setminus \bigcup_{i\leq j}D_i\pitchfork D_j$,
so we restrict our study to étale maps to $X$ on which $D$ is smooth.

Fix an extension $\ocE$. 
Let $x\in |D^\mathrm{sm}|$ and $U$ be an affine neighborhood of $x$ in the étale topology, on which the divisor $D$ is smooth
and $\ocE$ is trivial.
Shrinking $U$, we may choose an étale map $(z,x_1,\dotsc,x_d):U\ra \A^{d+1}_K$ such that $D\cap U=\{z=0\}$.
Fix a trivialization of $\ocE_U$. Then, in coordinates, $\nabla$ is of the form:
\begin{equation}\label{localform}
  \nabla = d + \frac{A}{z}dz +\sum_i B_idx_i\text{ with }A,B_1,\dotsc,B_d\in \mathcal{M}_r(\cO(U)).
\end{equation}
For each $y\in D\cap U$, the conjugacy class of $A(y)\in\End{\ocE(y)}$ is independent of the choice of coordinates $z$, $x_i$.
Moreover, one can deduce from flatness of $\nabla$ that its conjugacy class in $\cM_r(\Kbar)$ is independent of $y$ in
the connected components of $D\cap U$ (\cite[Appendix A]{andreRhamCohomologyDifferential2001}).
Hence we see that that to $\ocE$ we can associate the bundles $\ocE_i:=\ocE_{\mid D_i}$
together with endomorphisms $A_i\in\End{\ocE_i}$ whose conjugacy classes are constant.

We now introduce the notion of Deligne extension.
Let $\Gal_K$ be the Galois group of the extension $\Kbar/K$. Then as $K$ is of characteristic $0$,
the abelian group $\Kbar/\Gal_K$ is torsion-free. Moreover, it has a natural action of $\Z$ by translation, for which every orbit is free.

\begin{definition}[Deligne's $\tau$-extension]
  Let $(\cE,\nabla)$ be a bundle with flat connection on $X\setminus D$ and $\ocE$ an extension to $X$.
  Let $\tau=(\tau_1,\dotsc,\tau_n)$ be a family of sections $\tau_i:(\Kbar/\Gal_K)/\Z\ra \Kbar/\Gal_K$, one for each irreducible component of $D$.
  We say that $\ocE$ is a $\tau$-extension if for each $i$, the eigenvalues of $A_i$ as above are in the image of $\tau_i$.
\end{definition}

\begin{theorem}[{\cite[def. 4.4, thm. 4.9]{andreRhamCohomologyDifferential2001}}]\label{DeligneKbar}
  If $K$ is algebraically closed, for every $\tau$, a bundle with flat connection $(\cE,\nabla)$ on $X\setminus D$ regular with respect to $(X,D)$
  admits a $\tau$-extension, unique up to unique isomorphism.
\end{theorem}

We can use descent to remove the assumption that $K=\Kbar$.

\begin{corollary}
  For any $K$ of characteristic $0$ and any $\tau$, a bundle with flat connection $(\cE,\nabla)$ on $X\setminus D$ regular with respect to $(X,D)$
  admits a $\tau$-extension, unique up to unique isomorphism.
\end{corollary}
\begin{proof}
  Each $\tau_i$ lifts uniquely to a section $\Kbar/\Z\ra\Kbar$ that we will also denote $\tau_i$.
  Let us start by proving uniqueness. Assume $\ocE_1$ and $\ocE_2$ are $\tau$-extensions of $(\cE,\nabla)$. Then by uniqueness in \Cref{DeligneKbar},
  their pullbacks $\ocE_{1,\Kbar}$ and $\ocE_{2,\Kbar}$ to $X_{\Kbar}$ are uniquely isomorphic. Denote by $\phi$ this isomorphism.
  By uniqueness of $\phi$, for each $\sigma\in\Gal_K$, $\sigma^*\phi=\phi$. Hence $\phi$ descends to a unique $\phi_K:\ocE_1\simeq \ocE_2$.

  For existence, note that by \Cref{DeligneKbar} the pullback $(\cE_{\Kbar},\nabla_{\Kbar})$ to $(X\setminus D)_{\Kbar}$ admits a $\tau$-extension
  $\ocE_{\Kbar}$. This extension is definable over a finite extension $L/K$, that we may assume to be Galois. 
  Hence, for this $L$, $(\cE_{L},\nabla_{L})$ admits a $\tau$-extension $\ocE_L$. By uniqueness of the $\tau$-extension
  and because $\tau$ is $\Gal_K$-equivariant, for each $\sigma\in \Gal(L/K)$
  there is a canonical isomorphism $\sigma^*\ocE_L\simeq \ocE_L$. These isomorphisms provide descent data for the bundle $\ocE_L$,
  and hence by Galois descent, $\ocE_L$ is definable over $K$, say $\ocE_L=p^*\ocE$ with $p:X_L\ra X$.
  Now $\ocE$ is desired $\tau$-extension of $(\cE,\nabla)$.
\end{proof}


The uniqueness of the $\tau$-extension can be used to prove that regularity of a connection is étale-local,
and thus can be checked in local coordinates.

\begin{proposition}
  Let $(\cE,\nabla)$ be a flat bundle over $X\setminus D$. Then $(\cE,\nabla)$ is regular with respect to $(X,D)$
  if and only if it is on an étale cover of $X$.
\end{proposition}
\begin{proof}
  Let $(U^j\ra X)_j$ be an étale cover. Denote by $D^j$ the pullback to $D$ to $U^j$. Assume that $(\cE,\nabla)$
  is regular with respect to $(U^j,D^j)$ for all $j$. Choose any $\tau$. Then $(\cE,\nabla)$ admits on each $U^j$
  a unique $\tau$-extension $\ocE^j$. By uniqueness, these extensions glue on the $U^j\times_X U^{j'}$ to a $\tau$-extension of $(\cE,\nabla)$.
  So $(\cE,\nabla)$ is regular with respect to $(X,D)$.
\end{proof}


\subsection{Specialization to the normal bundle of a divisor}\label{subsectionspecialization}


We first define the specialization of a $\tau$-extension and then explain how the specialization is independent of the choice of $\tau$.
We assume that $D_1$ is smooth for notational simplicity. If $D_1$ has self-crossings, everything works, provided we replace $ND_1$
by the normal bundle $N\hat{D}_1$ to $\hat{D}_1\ra X$, where $\hat{D}_1$ denotes the normalization of $D_1$.

Let us first briefly review the $\mathcal{D}$-module formalism. A flat connection $\nabla$ on a bundle $\cE$ over $Y$
is the same as the structure, on $\cE$, of a module over the algebra $\mathcal{D}_Y$ of differential operators.
In étale coordinates $(y_k)_k$, the operator $\frac{\partial}{\partial y_k}$ acts by $\nabla_{\frac{\partial}{\partial y_k}}$,
and $\mathcal{D}_Y$ is generated by the $\frac{\partial}{\partial y_k}$.
The datum of $(\cE,\nabla)$ over $X\setminus D$ together with an extension $\ocE$ is the same as that of a structure, on $\ocE$,
of a module over the algebra $\mathcal{D}_{X,D}$ of differential operators preserving the ideal $\cI_D$ defining $D$.
In étale coordinates $(z,x_1,\dotsc,x_d)$ with $D=\{z=0\}$, the algebra $\mathcal{D}_{X,D}$ is generated by
$z\frac{\partial}{\partial z},\frac{\partial}{\partial x_1},\dotsc,\frac{\partial}{\partial x_d}$,
and acts on $\ocE$ via the action of $\mathcal{D}_{X\setminus D}$ on $\cE$ induced by $\nabla$.
The $\log D$ assumption ensures that it is well defined.


\begin{definition}\label{definitionsspecialization}
  Let $\cI_1$ be the ideal defining $D_1$. The normal bundle $ND_1$ of $D_1$ is defined as
  $\underline{\Spec}_{D_1}(\bigoplus_{k\geq 0}\cI_1^k/\cI_1^{k+1})$. It contains a divisor $D_{ND_1}=D_1\cup D_2'\cup\dotsb\cup D_n'$ induced by $D$.
  For a bundle with flat connection $(\cE,\nabla)$ on $X\setminus D$
  and a $\tau$-extension $\ocE_\tau$ to $X$. Define the bundle
  $$\Sp_{\tau}(\cE,\nabla):=\bigoplus_{k\geq 0}\cI_1^k\ocE_\tau/\cI_1^{k+1}\ocE_\tau.$$
  It is a free bundle on $ND_1$ and is naturally equipped with an action of the algebra $\mathcal{D}_{ND_1,D_{ND_1}}$.
  We will denote by $\nabla^{1,\tau}$ the corresponding connection.
\end{definition}

\begin{remark}\label{remarkspecialization}
  The action of $\mathcal{D}_{ND_1,D_{ND_1}}$ comes from the fact that:
  $$\mathcal{D}_{ND_1,D_{ND_1}}=\bigoplus_{k\geq 0}\cI_1^k\mathcal{D}_{X, D}/\cI_1^{k+1}\mathcal{D}_{X, D}.$$
  In coordinates as in \Cref{localform}, the specialization $\nabla^{1,\tau}$ takes the form:
  \begin{equation*}
    \nabla = d - \frac{A(z=0)}{z}dz -\sum_i B_i(z=0)dx_i.
  \end{equation*}
  Note that the definition of the specialization makes sense for any extension $\ocE$ of $(\cE,\nabla)$. However, in this generality,
  the specialization will depend on the choice of the extension. When restricting to $\tau$-extensions, as explained below,
  the specialization is independent of $\tau$.
\end{remark}

The connection $\nabla^{1,\tau}$ is $\mathbb{G}_m$-equivariant for natural line bundle structure on $ND_1$. This induces a natural decomposition
of $\Sp_{\tau}(\cE,\nabla)$ that we now explain. Note that $\mathcal{D}_{ND_1,D_{ND_1}}$ is graded and the degree $0$ piece canonically
contains the sheaf theoretic pullback of $\left(\Omega^1_X(\log D)_{\mid D_1}\right)^\vee$ to $ND_1$ as the linear subspace
spanned by vector fields.
Hence the residue map $\mathrm{res}:\Omega^1_X(\log D)_{\mid D_1}\ra \cO_{D_1}$ induces a canonical map $\cO_{D_1}\ra (\mathcal{D}_{ND_1,D_{ND_1}})_0$.
Let us denote by $r$ the vector field which is the image of $1$ by this map.
In étale coordinates $(z,x_1,\dotsc,x_d)$ with $D_1=\{z=0\}$, $r$ is just the class of $z\frac{\partial}{\partial z}$ in
$(\mathcal{D}_{ND_1,D_{ND_1}})_0=\mathcal{D}_{X,D}/\cI_1\mathcal{D}_{X,D}$.

For $\alpha\in \Kbar/\Gal_K$ in the image of $\tau_1$, denote by $\ocE_\alpha\subset \Sp_\tau(\cE)$ the subspace spanned by the generalized
eigenspaces of $r$ for the eigenvalues $\lambda\in \Kbar$ such that $\lambda-m\in \alpha$ for some $m\in \Z$.
In local coordinates such as in \Cref{remarkspecialization}, $\ocE_{\tau,\alpha}$ is the subbundle of $\ocE$ corresponding to the generalized eigenspaces
of $A(z=0)$ with eigenvalues in $\alpha$. As the conjugacy class of $A(z=0)$ is constant, we get the following.

\begin{proposition}\label{propositiondecomposition}
  We have a bundle decomposition compatible with the connection $\Sp_\tau(\cE,\nabla)=\bigoplus_{\alpha\in\myim{\tau}}\ocE_{\tau,\alpha}$.
\end{proposition}

\begin{lemma}\label{lemmaspecialization}
  For any choice of $\tau$, $\tau'$, and pair $\alpha,\alpha'\in \Kbar/\Gal_K$ such that $\alpha-\alpha'\in\Z$, $\alpha\in\myim{\tau_1}$ and $\alpha'\in\myim{\tau_1'}$,
  there is a canonical isomorphism $\ocE_{\tau,\alpha\mid U_{ND_1}}\simeq \ocE_{\tau',\alpha'\mid U_{ND_1}}$, where $U_{ND_1}$ denotes $ND_1\setminus D_{ND_1}$.
\end{lemma}

When $X$ is a curve and $K$ is algebraically closed, this Lemma is a direct consequence of Manin's classification of Fuchsian systems \cite{maninModuliFuchsiani1965}.
For a statement in english, see \cite[3.2.2]{andreRhamCohomologyDifferential2001}. The proof we give below is by reduction to Manin's result.

\begin{proof}
  From the definition of the specialization, one deduces that $\Sp_\tau(\cE,\nabla)_{\mid U_{ND_1}}$ is independent of the choice of $\tau_i$ for $i\geq 2$.
  Hence we restrict our attention to $\tau_1$ and $\tau_1'$. Note that by uniqueness, for any $k\in \Z$, the $(\tau_1+k,\tau_2,\dotsc,\tau_n)$-extension
  of $(\cE,\nabla)$ is $\ocE_\tau\otimes \cO(kD_1)$. Hence $\Sp_{(\tau_1+k,\tau_2,\dotsc,\tau_n)}(\cE,\nabla)\simeq \Sp_\tau(\cE,\nabla)\otimes \cO(kD_1)$
  and $\Sp_{(\tau_1+k,\tau_2,\dotsc,\tau_n)}(\cE,\nabla)_{\mid U_{ND_1}}\simeq \Sp_\tau(\cE,\nabla)_{\mid U_{ND_1}}$ canonically.
  So we may assume that $\alpha=\alpha'$.

  The free action of $\Z$ on $\Kbar/\Gal_K$ defines a partial order by $a\leq b$ if $b-a\in \N$. Choose $\tau''$ such that for all $i$
  $\tau_i''\leq \tau_i$ and
  $\tau_i''\leq \tau_i$, and $\alpha=\alpha'\in \myim{\tau_1''}$. We now construct isomorphisms
  $\ocE_{\tau'',\alpha\mid U_{ND_1}}\simeq \ocE_{\tau,\alpha\mid U_{ND_1}}$ and $\ocE_{\tau'',\alpha\mid U_{ND_1}}\simeq \ocE_{\tau',\alpha\mid U_{ND_1}}$
  such that the composition $\ocE_{\tau,\alpha\mid U_{ND_1}}\simeq \ocE_{\tau',\alpha'\mid U_{ND_1}}$ is independent of the choice of $\tau''$.
  Both $\ocE_\tau$ and $\ocE_{\tau''}$ are subsheaves of $j_*\cE$ where $j:X\setminus D\ra X$ is the inclusion.
  We claim that $\tau_1''\leq \tau_1$ implies $\ocE_\tau\subset \ocE_{\tau''}$. Indeed, this may be checked over $\overline{K}$ on curves mapping
  to $X$ transversely to $D$,
  where it is then a direct consequence of Manin's classification (see \cite[3.2.2]{andreRhamCohomologyDifferential2001}).
  This leads to a map $\Sp_\tau(\cE,\nabla)\ra\Sp_{\tau''}(\cE,\nabla)$ and hence a map
  $\phi^{\alpha}_{\tau,\tau''}:\ocE_{\tau,\alpha\mid U_{ND_1}}\ra \ocE_{\tau'',\alpha\mid U_{ND_1}}$. Again, this map can be checked to be an isomorphism
  over $\overline{K}$ on every curve over $X$ transverse to $D$ by Manin's classification. Thus $\phi^{\alpha}_{\tau,\tau''}$ is an isomorphism.
  To check that $(\phi^{\alpha}_{\tau',\tau''})^{-1}\circ \phi^{\alpha}_{\tau,\tau''}$ is independent of $\tau''$,
  we need only check that for any $\tau^{(1)},\tau^{(2)},\tau^{(3)}$
  such that $\tau^{(1)}_i\leq \tau^{(2)}_i\leq \tau^{(3)}_i$ for all $i$ and $\alpha\in \myim{\tau^{(1)}_1}\cap\myim{\tau^{(2)}_1}\cap\myim{\tau^{(3)}_1}$,
  $\phi^{\alpha}_{\tau^{(2)},\tau^{(1)}}\circ \phi^{\alpha}_{\tau^{(3)},\tau^{(2)}}=\phi^{\alpha}_{\tau^{(3)},\tau^{(1)}}$. Again, this can be checked
  over $\overline{K}$ on curves.
\end{proof}


As a direct corollary of \Cref{lemmaspecialization}, we get that specialization is, over $U_{ND_1}$, independent of the choice of $\tau$.

\begin{propositiondefinition}
  Let $(\cE,\nabla)$ be a bundle with flat connection on $X\setminus D$, regular with respect to $(X,D)$.
  For any $\tau$, $\tau'$,
  there is a canonical isomorphism $\Sp_\tau(\cE,\nabla)_{\mid U_{ND_1}}\simeq\Sp_{\tau'}(\cE,\nabla)_{\mid U_{ND_1}}$.
  This defines a bundle with flat connection $\Sp(\cE,\nabla)$ on $U_{ND_1}$, regular with respect to $(ND_1,D_{ND_1})$,
  that we call the \emph{specialization} of $(\cE,\nabla)$ to $U_{ND_1}$.
\end{propositiondefinition}

\begin{remark}\label{remarkequivarianceappendix}
  Note that $U_{ND_1}$ is a principal $\mathbb{G}_m$-bundle over $D_1$ and that for any $(\cE,\nabla)$, $\Sp(\cE,\nabla)$ is by definition
  $\mathbb{G}_m$-equivariant. This justifies \Cref{remarkequivariance}.
\end{remark}


\subsection{Behavior of regular singular connections with respect to root stacks}


In the last part of this appendix, we prove the statements of \Cref{sectionBK} on the behavior of regularity and specialization
with respect to taking root stack. For the convenience of the reader, we repeat the statements of the $2$ results we prove here.


\begin{lemma*}[{\ref{regularrootstack}}]
  Let $X$ be a geometrically connected smooth separated Deligne-Mumford stack over a field $K$ of characteristic $0$.
  Let $D=D_1\cup\dotsb\cup D_n\subset X$ be a normal crossing divisor with each $D_i$ geometrically connected, possibly with normal self-crossings,
  and $r_i$, $i=1,\dotsc,n$ positive integers.
  Let $X':=X[\sum_i\frac{D_i}{r_i}]$ be the root stack, $D'=\bigcup_i\frac{D_i}{r_i}$ the divisor and $f:X'\ra X$ the projection.
  Then a connection on a bundle over $X\setminus D=X'\setminus D'$ is regular for $(X,D)$ if and only if it is regular for $(X',D')$.
  Moreover, for $(\cE,\nabla)$ regular, if $\overline{\cE}$ is an extension to $X$,
  then $f^*\overline{\cE}$ is an extension to $X'$, and if $\overline{\cE}'$ is an extension to $X'$,
  then $f_*\overline{\cE}'$ is an extension to $X$.
\end{lemma*}

\begin{proof}[Proof of \Cref{regularrootstack}]
  Both statements about $f^*\ocE$ and $f_*\ocE'$ are étale local on $X$.
  Hence we are reduced to the situation where $X=\Spec A$ is an affine smooth $K$-scheme,
  and $D=\bigcup_iD_i$ with $D_i$ smooth and the zero locus of some $x_i\in A$.
  Then $\widetilde{X}:=\Spec A[y_i\:|\: i=1,\dotsc, d]/(y_i^{r_i}-x_i)$ is a ramified covering $p:\widetilde{X}\ra X$ with Galois group $G=\prod_i\mu_{r_i}$,
  and $\widetilde{X}/G$ is isomorphic to $X'=X[\sum_i\frac{D_i}{r_i}]$. Denote by $\widetilde{D}$ the union of the $V(y_i)$.
  The category of connections on $(X'=\widetilde{X}/G,D')$
  is equivalent to that of $G$-equivariant connections on $(\widetilde{X},\widetilde{D})$, and extensions correspond to equivariant extensions.

  Let $\ocE$ be an extension of $(\cE,\nabla)$ to $(X,D)$. Then, as $p^*\Omega^1(\log D)\simeq \Omega^1(\log \widetilde{D})$,
  $p^*\ocE$ is a $G$-equivariant extension of the pullback $(p^*\cE,p^*\nabla)$. As $p^*\ocE$ corresponds to $f^*\ocE$, the latter is an extension
  of $(\cE,\nabla)$ to $(X',D')$.

  Let $\ocE'$ be an extension of $(\cE,\nabla)$ to $(X',D')$. It corresponds to a $G$-equivariant extension $\ocE'_{\widetilde{X}}$ of $(q^*\cE,q^*\nabla)$,
  where $q:\widetilde{X}\ra X'$ is the quotient map.
  Let $\widetilde{A}=\Gamma(\widetilde{X},\cO_{\widetilde{X}})$ and $M=\Gamma(\widetilde{X},\ocE'_{\widetilde{X}})$.
  Then $f_*\ocE'$ is the quasi-coherent sheaf associated to the $A$-module $M^G$. Let us first show that $M^G$ is a projective module,
  whose rank over $A$ is that of $M$ over $\widetilde{A}$.
  We are in characteristic $0$, so the map $m\mapsto \frac{1}{|G|}\sum_ggm$ provides a retraction $M\ra M^G$ of $A$-modules.
  As $M$ is a projective $A$-module, so is $M^G$.
  Hence $f_*\ocE'$ is a bundle.
  Now, as $X'\setminus D'\ra X\setminus D$ is an isomorphism, $(f_*\ocE')_{\mid X\setminus D}$ has the same rank as $\ocE'$.

  Let us now prove that $f_*\ocE'$ is an extension of $(\cE,\nabla)$ to $(X,D)$. We need only show that for $m\in M^G$ and each $i=1,\dotsc,d$,
  $\nabla_{x_i\frac{\partial}{\partial x_i}}(m)\in (M^G)_{x_i}$ lies in $M^G$. However, as $\frac{dx_i}{x_i}=r_i\frac{dy_i}{y_i}$,
  $x_i\frac{\partial}{\partial x_i}=\frac{1}{r_i}y_i\frac{\partial}{\partial y_i}$ and $\nabla_{x_i\frac{\partial}{\partial x_i}}(m)$ equals
  $\frac{1}{r_i}\nabla_{y_i\frac{\partial}{\partial y_i}}(m)$, which is in $M$ by regularity of $p^*\nabla$, and is $G$-invariant. Hence the result.
\end{proof}

\begin{lemma*}[{\ref{specializationrootstack}}]
  Let $(X,D)$ and $(X',D')$ be as in \Cref{regularrootstack}. Let $\hat{D}_1$ be the normalization\footnote{As $D_1$ may have self-crossings, it is not necessarily smooth.}
  of $D_1$ and for $i\geq 1$, $\hat{D}_{1i}$ the pullback of $D_i\pitchfork D_1$ to $\hat{D}_1$.
  Similarly define $\hat{D}_1'$ and $\hat{D}_{1i}'$ for the pair $(X',D')$.
  Then specialization commutes with root stack, in the sense that the following diagram commutes:
  \[\begin{tikzcd}
    {\RS{X}{D}} && {\RS{N\hat{D}_1}{\hat{D}_1\cup N\hat{D}_{1\mid\bigcup_i\hat{D}_{1i}}}} \\
    {\RS{X'}{D'}} && {\RS{N\hat{D}_1'}{\hat{D}_1'\cup N\hat{D}'_{1\mid\bigcup_i\hat{D}_{1i}'}}}
    \arrow["{\text{specialization}}", from=1-1, to=1-3]
    \arrow[tail reversed, from=1-1, to=2-1]
    \arrow[tail reversed, from=1-3, to=2-3]
    \arrow["{\text{specialization}}", from=2-1, to=2-3]
  \end{tikzcd}\]
  where the vertical equivalences of categories are provided by \Cref{regularrootstack}.
\end{lemma*}

\begin{proof}[Proof of \Cref{specializationrootstack}]
  First note that the normal bundle in the root stack is the root stack of the normal bundle.
  We will prove that push-forward along root stacks commutes with specialization maps.
  As a first step, let us prove that if $\ocE'_{\tau'}$ is the $\tau'$-extension of $(\cE,\nabla)$ to $X'$,
  then $f_*\ocE'_{\tau'}$ is the $\tau$-extension of $(\cE,\nabla)$ to $X$ for some $\tau$.
  To do so, we need only look at eigenvalues of residues locally. Indeed, an extension is a $\tau$-extension for some $\tau$
  if and only if on each connected component of $D$, the eigenvalues of the residue of the connection do not differ by non-zero integers.
  Let us place ourselves in the situation of the proof of \Cref{regularrootstack},
  and use the same notations.
  The group $G=\prod_i\mu_{r_i}$ has characters $\chi_{\ua}$ indexed by $\ua=(a_1,\dotsc,a_d)$ with $a_i\in\{0,\dotsc,r_i-1\}$,
  such that $\chi_{\mathrm{1}_j}:\prod_i\mu_{r_i}\ra \mu_{r_j}\subset \Kbar^\times$ is the projection on $\mu_{r_j}$.
  Now, after shrinking $U$, the module $M$ decomposes as $\bigoplus_{\ua}\chi_{\ua}\otimes \widetilde{A}^{\oplus n_{\ua}}$ for the action of $G$.
  As the connection on $M$ is $G$ equivariant, it is compatible with this decomposition and is thus of the form:
  \begin{equation*}
    d+\sum_{i}\sum_{b} B_{i,b}\frac{dy_i}{y_i} + \Omega\text{ with }B_{i,b}\in \End{\bigoplus_{\ua\mid a_i=b}\widetilde{A}^{\oplus n_{\ua}}}
    \text{ and }\Omega\text{ regular.}
  \end{equation*}
  A computation yields that $M^G$ is then $\bigoplus_{\ua}(\prod_iy_i^{a_i}) A^{\oplus n_{\ua}}$ with connection:
  \begin{equation*}
    d+\sum_{i}\sum_{b} \frac{b+B_{i,b}}{r_i}\frac{dx_i}{x_i} + \Omega'\text{ with }\Omega'\text{ regular.}
  \end{equation*}
  Now by another small computation, for each $i$, as the eigenvalues of $\bigoplus_b B_{i,b}$ do not differ by non-zero integers, and $0\leq b\leq r_i-1$,
  neither do the eigenvalues of $\bigoplus_b\frac{b+B_{i,b}}{r_i}$. Hence $M^G$ is a Deligne $\tau$-extension for some $\tau$.
  More precisely, $\tau$ is determined by $\myim{\tau_i}=\frac{1}{r_i}\left(\myim{\tau_i'}+\{0,1,\dotsc,r_i-1\}\right)$.
  In particular, $\myim{\tau_1}=\sqcup_{0\leq b\leq r_1-1}\Lambda_b$ with $\Lambda_b=\frac{\myim{\tau_1'}}{r_1}+\frac{b}{r_1}$.
  This allows us to apply \Cref{propositiondecomposition} to get a decomposition:
  \begin{equation*}
    \Sp_{D_1,\tau}(\cE,\nabla)=\bigoplus_{0\leq b\leq r_1-1} \Sp_{D_1,\tau}(\cE,\nabla)_b.
  \end{equation*}

  The second step is to define a natural isomorphism $\bar{f}_*\Sp_{D_1',\tau'}(\cE,\nabla)\simeq \Sp_{D_1,\tau}(\cE,\nabla)$,
  where $\bar{f}:N\hat{D}_1'\ra N\hat{D}_1$.
  To do so, denote by $\cI$ the ideal defining $D_1$ in $X$ and by $\cI'$ the ideal defining $D_1'$ in $X'$.
  We may work étale locally and then glue the isomorphisms. Hence we may assume that $D_1$ is smooth.
  As $D_1'$ is a $\mu_{r_1}$-gerbe over $D_1$, each fiber of $\ocE'_{\tau'}/\cI'\ocE'_{\tau'}$
  has an action of $\mu_{r_1}$. Hence we get a decomposition of $\ocE'_{\tau'}/\cI'\ocE'_{\tau'}$ according to characters for this action.
  Multiplying the decomposition by $\cI'^k$ for all $k\geq 0$, we get a decomposition compatible with the connection:
  \begin{equation*}
    \Sp_{D_1',\tau'}(\cE,\nabla)=\bigoplus_{0\leq b\leq r_1-1} \Sp_{D_1',\tau'}(\cE,\nabla)_b.
  \end{equation*}
  We now describe natural isomorphisms $\bar{f}_*\left(\Sp_{D_1',\tau'}(\cE,\nabla)_b\right)\simeq \Sp_{D_1,\tau}(\cE,\nabla)_b$.
  By flatness of $f$, we have an isomorphism:
  \begin{align*}
    \Sp_{D_1,\tau}(\cE,\nabla)=\bigoplus_{k\geq 0}\cI^kf_*\ocE'_{\tau'}/\cI^{k+1}f_*\ocE'_{\tau'}&\simeq f_{1*}\left(\bigoplus_{k\geq 0}f^*(\cI)^k\ocE'_{\tau'}/f^*(\cI)^{k+1}\ocE'_{\tau'}\right)\\
                                                                      &=f_{1*}\left(\bigoplus_{k\geq 0}\cI'^{rk}\ocE'_{\tau'}/\cI'^{rk+r}\ocE'_{\tau'}\right)
  \end{align*}
  where $f_1:X'\times_X D_1\ra D_1$ is the projection.
  Consider the natural maps:
  \[\begin{tikzcd}
    {\Sp_{D_1,\tau}(\cE,\nabla)_b} & {f_{1*}\left(\bigoplus_{k\geq 0}\cI'^{rk+b}\ocE'_{\tau'}/\cI'^{rk+r}\ocE'_{\tau'}\right)\cap \Sp_{D_1,\tau}(\cE,\nabla)_b} \\
    {\bar{f}_*\left(\Sp_{D_1',\tau'}(\cE,\nabla)_b\right)} & {f_{1*}\left(\left(\bigoplus_{k\geq 0}\cI'^{rk+b}\ocE'_{\tau'}/\cI'^{rk+b+1}\ocE'_{\tau'}\right)\cap \Sp_{D_1',\tau'}(\cE,\nabla)_b\right).}
    \arrow[from=1-2, to=1-1]
    \arrow[from=1-2, to=2-2]
    \arrow[from=2-2, to=2-1]
  \end{tikzcd}\]
  One can check locally in coordinates as in the first part of this proof that all $3$ maps in the diagram above are isomorphisms.
  Hence the desired isomorphism $\bar{f}_*\left(\Sp_{D_1',\tau'}(\cE,\nabla)_b\right)\simeq \Sp_{D_1,\tau}(\cE,\nabla)_b$
  and summing over $b$, the sought isomorphism $\bar{f}_*(\Sp_{D_1',\tau'}(\cE,\nabla))\simeq \Sp_{D_1,\tau}(\cE,\nabla)$.
\end{proof}



\bibliographystyle{plain}
\bibliography{biblio}

\end{document}

%% file: tikzstyles.tikzstyles
\tikzstyle{green}=[text={black!30!green}]
\tikzstyle{blue}=[text=blue]
\tikzstyle{red}=[text=red]
\tikzstyle{puncture}=[fill=white, draw=red, shape=circle, minimum size=1pt]
\tikzstyle{blackpuncture}=[fill=white, draw=black, shape=circle, minimum size=1pt]
\tikzstyle{cyan}=[text=cyan]

\tikzstyle{markings}=[-, draw=red, line width=1pt, line cap=round]
\tikzstyle{overbraid}=[-, draw=white, fill=none, line width=6pt]
\tikzstyle{thick}=[-, line width=2pt, draw=blue]
\tikzstyle{dashedline}=[-, dashed]
\tikzstyle{dottedline}=[-, dash pattern=on 0.75pt off 0.75pt, line width=0.75pt]
\tikzstyle{thin red}=[-, line width=0.25pt, draw=black]
\tikzstyle{tangle}=[-, draw=blue, line width=1pt, fill={blue!20}]
\tikzstyle{scc}=[-, draw={black!30!green}, fill={blue!20}, line width=1pt]
\tikzstyle{inner boundary}=[-, fill=white]
\tikzstyle{outer boundary}=[-, fill={red!20}]
\tikzstyle{lowerboundery}=[-, line width=1.5pt, line cap=round, draw=red]
\tikzstyle{upperboundery}=[-, line width=1.5pt, line cap=round, draw=blue]
\tikzstyle{dottedcycle}=[-, draw=blue, dash pattern=on 0.5pt off 1pt on 4pt off 1pt, decoration={markings, mark=at position 0.5 with {\arrow{>}}}, postaction=decorate]
\tikzstyle{cycle}=[-, draw=blue, decoration={markings, mark=at position 0.5 with {\arrow{>}}}, postaction=decorate]
\tikzstyle{path}=[-, draw=cyan, line width=0.25pt]
\tikzstyle{arrowpath}=[-, draw=cyan, line width=0.25pt, decoration={markings, mark=at position 0.5 with {\arrow{>}}}, postaction=decorate]
\tikzstyle{orientedpath}=[-, line width=0.25pt, decoration={markings, mark=at position 0.5 with {\arrow{<}}}, postaction=decorate]
\tikzstyle{inner square}=[-, fill={blue!20}]
\tikzstyle{outer square}=[-, fill={red!20}]
\tikzstyle{blueline}=[-, draw=blue]
\tikzstyle{greenline}=[-, draw=green]
\tikzstyle{bluesquare}=[-, draw=blue, fill={blue!20}]